\documentclass[12pt,reqno,a4paper]{amsart}
\usepackage{
    amsmath,  amsfonts, amssymb,  amsthm,   amscd,
    gensymb,  graphicx, comment,  etoolbox, url,
    booktabs, stackrel, mathtools,enumitem, mathdots,  microtype, lmodern,    mathrsfs, graphicx, tikz,  longtable,tabularx, float, tikz, pst-node, tikz-cd, multirow, tabularx, amscd,  bm, array, makecell, diagbox, booktabs,ragged2e, caption, subcaption }
\usepackage{makecell,slashbox}

\usepackage{thmtools} 

\usepackage{xcolor}
\usepackage[all]{xy}
\usepackage{graphicx}

\usepackage[utf8]{inputenc}
\usepackage{microtype, fullpage, wrapfig,textcomp,mathrsfs,csquotes,fbb}
\usepackage[pagebackref=true, colorlinks=true, linkcolor=blue, citecolor=blue, urlcolor=blue, breaklinks=true]{hyperref}
\usepackage[capitalise]{cleveref}

\usepackage{todonotes}
\usetikzlibrary{positioning}
\usetikzlibrary{shapes,arrows.meta,calc}
\usetikzlibrary{arrows}

\makeatletter
\def\@tocline#1#2#3#4#5#6#7{\relax
  \ifnum #1>\c@tocdepth 
  \else
    \par \addpenalty\@secpenalty\addvspace{#2}%
    \begingroup \hyphenpenalty\@M
    \@ifempty{#4}{%
      \@tempdima\csname r@tocindent\number#1\endcsname\relax
    }{%
      \@tempdima#4\relax
    }%
    \parindent\z@ \leftskip#3\relax \advance\leftskip\@tempdima\relax
    \rightskip\@pnumwidth plus4em \parfillskip-\@pnumwidth
    #5\leavevmode\hskip-\@tempdima
      \ifcase #1
       \or\or \hskip 1em \or \hskip 2em \else \hskip 3em \fi%
      #6\nobreak\relax
    \hfill\hbox to\@pnumwidth{\@tocpagenum{#7}}\par
    \nobreak
    \endgroup
  \fi}
\makeatother

\newcommand{\mcat}{\mathrm{cat}_m}

\newtheorem{theorem}{Theorem}[section]

\newtheorem{corollary}[theorem] {Corollary}
\newtheorem{lemma}[theorem]{Lemma}

\newtheorem{proposition}[theorem]{Proposition}

\theoremstyle{definition}
\newtheorem{definition}[theorem]{Definition}
\newtheorem{example}[theorem]{Example}
\newtheorem{remark}[theorem]{Remark}

\newtheorem{property}{Property}[]

\newcommand\id{\operatorname{id}}
\newcommand\hdim{\operatorname{hdim}}

\newcommand{\Dm}{\operatorname{D_m}}
\newcommand{\Dn}{\operatorname{D_n}}

\newcommand\HD{\operatorname{H^{*}D}}
\newcommand\HDn{\operatorname{H^{*}D_{n}}}
\newcommand\HDm{\operatorname{H^{*}D_{m}}}
\newcommand\HDmB{\operatorname{H^*D_{m,B}}}

\newcommand{\TC}{\operatorname{TC}}
\newcommand\TCm{\operatorname{TC^m}}

\newcommand{\conn}{\operatorname{conn}}

\newcommand{\Hom}{\operatorname{Hom}}
\newcommand{\cl}{\mathrm{cl}}

\newcommand{\ct}{\mathrm{cat}}

\newcommand{\pr}{\mathrm{pr}}

\newcommand\Q{\mathbb{Q}}
\newcommand\R{\mathbb{R}}
\newcommand\Z{\mathbb{Z}}

\newcommand{\J}{\mathcal{J}}

\newcommand{\sct}{\mathrm{secat}}
\newcommand{\sctm}{\mathrm{secat}_{m}}

\newcolumntype{x}[1]{>{\centering\arraybackslash}p{#1}}

\begin{document}
\title{On the \lowercase{m}-dimensional sectional category and induced invariants}
\author[R. Singh]{Ramandeep Singh Arora}
\address{Department of Mathematics, Indian Institute of Science Education and Research Pune, India}
\email{ramandeepsingh.arora@students.iiserpune.ac.in}
\email{ramandsa@gmail.com}
\author[S. Datta]{Sutirtha Datta}
\address{Department of Mathematics, Indian Institute of Science Education and Research Pune, India}
\email{sutirtha2702@gmail.com}
\author[N. Daundkar]{Navnath Daundkar}
\address{Department of Mathematics, Indian Institute of Technology Madras, Chennai, India.}
\email{navnath@iitm.ac.in}
\author[G. C. Dutta]{Gopal Chandra Dutta}
\address{Statistics and Mathematics Unit, Indian Statistical Institute, Kolkata-700108, India.}
\email{gdutta670@gmail.com}
	
\begin{abstract}
In this paper, we systematically study the $m$-dimensional sectional category of a fibration, introduced by Schwarz as an approximating invariant for the sectional category. 
We develop the basic theory of this invariant, establish its fundamental properties, and show how it gives rise to a hierarchy of induced invariants, including the $m$-dimensional Lusternik--Schnirelmann category, the $m$-topological complexity, and the $m$-homotopic distance between maps. 
We further investigate the relationships between these $m$-dimensional invariants and their classical analogues, present a variety of examples in which these invariants are computed, and illustrate when they agree with or differ from their classical counterparts. 
We also introduce the notion of $m$-cohomological distance and study its interaction with the $m$-homotopic distance.
Finally, we prove a Bochner-type theorem for $\mathrm{secat}_1$, extending the corresponding theorem of Oprea and Strom for $\mathrm{cat}_1$. 
We also establish a $\mathrm{cat}_m$ version of Oprea's improvement of Bochner's theorem.
\end{abstract}

\maketitle


\tableofcontents

\section{Introduction}
The notion of sectional category of a fibration $p\colon E \to B$ was introduced by Schwarz \cite{schwarz1961genus} as a generalization of the Lusternik--Schnirelmann category \cite{L-S-cat} of a space. 
The sectional category of $p$, denoted by $\sct(p)$, is the least integer $k$ such that $B$ admits an open cover $\{U_i\}_{i=0}^k$ for which each $U_i$ admits a continuous section of $p$. 
This notion was further developed by Berstein and Ganea in~\cite{B-G}, and subsequently by James in~\cite{James}. 
Since its introduction, sectional category has found deep connections with classical problems in homotopy theory, as well as modern applications in areas such as topological robotics.

Important special cases of the sectional category include the topological complexity $\TC(X)$ and the Lusternik--Schnirelmann category $\ct(X)$ of a space $X$.
Topological complexity was introduced by Farber in \cite{Farber-TC}. 
It can be interpreted as the sectional category of the free path space fibration $\pi_X \colon PX \to X \times X$, defined by $\pi_X(\gamma) = (\gamma(0), \gamma(1))$, where $PX$ denotes the free path space endowed with the compact--open topology. 
On the other hand, for a topological space $X$ with a chosen base point $x_0$, the Lusternik--Schnirelmann category $\ct(X)$ can be interpreted as the sectional category of the based path space fibration $e_X \colon P_{x_0}X \to X$, defined by $e_X(\gamma) = \gamma(1)$, where $P_{x_0}X$ denotes the based path space, also equipped with the compact--open topology.
Equivalently, $\ct(X)$ is the least integer $k$ such that $X$ admits an open cover $\{U_i\}_{i=0}^k$ for which each inclusion $U_i \hookrightarrow X$ is nullhomotopic. 
Variants and refinements of the sectional category have further led to the study of the homotopic distance between maps, equivariant invariants, and parametrized versions of motion planning complexity.

In his original work, Schwarz did not restrict attention solely to the absolute invariant $\sct(p)$, but also introduced the notion of the \emph{$m$-dimensional genus}, providing a finite-dimensional approximation to the sectional category. 
Motivated by this idea, we introduce the $m$-dimensional sectional category $\sct_m(p)$, which is shown to coincide with the $m$-dimensional genus, as an approximating invariant for $\sct(p)$. 
The sequence
$$
    \sct_1(p) \leq \sct_2(p) \leq \cdots \leq \sct(p)
$$
captures increasingly refined information about the fibration, stabilizing to the classical sectional category under suitable conditions.

The first aim of this paper is to develop a systematic study of $\sct_m$, including basic properties, fibre-preserving homotopy invariance, and functorial behavior. 
We show that $\sct_m$ naturally induces a hierarchy of $m$-dimensional invariants extending several classical notions. 
In particular, we  study the $m$-dimensional Lusternik--Schnirelmann category, the $m$-topological complexity, and the $m$-homotopic distance between maps. 
These invariants interpolate between low-dimensional computable approximations and their full homotopy-theoretic counterparts.

Our second goal is to analyze how these $m$-dimensional invariants relate to known bounds and cohomological methods. 
We show that many classical lower and upper bounds admit natural refinements at the $m$-dimensional level, leading to sharper estimates in concrete examples. 

Recently, Ekiz, Kuanyshov, and Borat introduced the notion of $m$-contiguity distance, a combinatorial analogue of $m$-homotopic distance for simplicial maps, providing a discrete framework for approximating homotopy-theoretic invariants. Interested readers are referred to \cite{yazici2026m} for further details.

\subsection{$m$-sectional category and $m$-category}
In \Cref{sec: msecat}, we develop the basic theory of the $m$-sectional category of a fibration. 
For a fibration $p \colon E \to B$, the \emph{$m$-sectional category} of $p$, denoted by $\sct_m(p)$, is the smallest integer $k$ such that the base space $B$ admits an open cover by $k+1$ sets with the following property: for an open set $U$ in the cover, every map $\phi \colon P \to U$ from an $m$-dimensional CW complex $P$ admits a map $s \colon P \to E$ such that $p \circ s = \iota_U \circ \phi$, where $\iota_U \colon U \hookrightarrow B$ denotes the inclusion.
We show in \Cref{prop: fibre-preserving homotopy invariance} that $\sct_m(p)$ is a fibre-preserving homotopy invariant of a fibration $p$.
We establish the cohomological lower bound on $\sct_m$ in \Cref{prop: cohomological lower bound of m-secat}. 
In particular, for a fibration $p\colon E\to B$, suppose there exist cohomology classes $u_0,\ldots,u_k$ of $B$, with coefficients in a commutative ring $R$, such that for each $i$, we have $p^*(u_i)=0$, but their cup product is nonzero. 
If the largest degree among these classes is $m$, then 
\[
    \sct_m(p)\geq k+1.
\]
We also establish the product inequality for $\sct_m$ in \Cref{prop: prod-ineq for m-secat}.

In \Cref{sec: m-category}, we recall the notion of the $m$-category, denoted by $\mcat(X)$, of a topological space $X$.
Then, in \Cref{subsec: m-category of map}, we extend this notion to maps and explore, in \Cref{prop: secatm and catmf}, the relationship between the $m$-sectional category of the pullback of a fibration $p \colon E \to B$ along a map $f \colon B' \to B$ and the invariant $\ct_m(f)$.
As an application of this result, we present \Cref{Example 1}, \Cref{Example 2}, and \Cref{Example 3} which illustrate both strict inequalities and equalities among $\sct_m$, $\sct_n$, and $\sct$ for different values of $m$ and $n$.
In particular, for the free path space fibration $\pi_{S^n} \colon P(S^n) \to S^n \times S^n$, with $n \geq 2$, we show in \Cref{Example 1} that 
$$
    \sct_m(\pi_{S^n})=0 \text{ for } m\leq n-1, \text{ whereas } \sct(\pi_{S^n})=\TC(S^n)\neq 0.
$$ 
For a nontrivial principal $K(G,1)$-bundle $p \colon E \to B$, where $G$ is a discrete abelian group and $B$ is a simply connected CW complex, we show in \Cref{Example 2} that 
$$
    \sct_1(p)=0, \text{ while } \sct_{m}(p)=\sct(p)\neq 0 \text{ for all } m\geq 2.
$$ 
For the universal covering map $p\colon S^n\to \R P^n$, we show in \Cref{Example 3} that 
$$
    \sct_m(p) = \sct(p)=n \text{ for all } m.
$$
Further, by observing in \Cref{cor: m-cat = m-secat(e_X)} that $\ct_m(X)=\sct_m(e_X)$, we establish the product inequality for $\ct_m$ in \Cref{prop: prod-ineq for m-cat} and a cohomological lower bound for $\ct_m$.
In particular, for a path-connected topological space $X$, we show in \Cref{prop: cohomological lower bound on m-cat} that, if $u_0,\ldots,u_k \in \widetilde{H}^*(X;R)$ have nonzero cup product and $m=\max\{\deg(u_i)\}$, then 
$$
    \mcat(X)\ge k+1.
$$
We also obtain, in \Cref{prop: m-cat = cat if X is ''aspherical''}, a sufficient condition on the space $X$ under which $\ct_m(X)=\ct(X)$.
We conclude the section by computing the $m$-category of Moore spaces and products of spheres in \Cref{ex: m-cat of Moore spaces} and \Cref{ex: m-cat of product of spheres}, respectively.
More precisely, the $m$-category of the Moore space $M(G,n)$ is given by
$$
\mcat(M(G,n)) =
\begin{cases}
    0 & \text{ if } m<n,\\
    1 & \text{ if } m\geq n,
\end{cases}
$$
and the $m$-category of the product of spheres $S^{\overline{n}} = S^{n_1} \times S^{n_2} \times \dots \times S^{n_k}$, where $n_1 < n_2 < \dots < n_k$, is given by
$$
\mcat(S^{\overline{n}}) =
\begin{cases}
    0 & \text{ if } m<n_1,\\
    i & \text{ if } n_i \leq m < n_{i+1},\\
    k & \text{ if } m \geq n_{k}.
    \end{cases}
$$

\subsection{$m$-homotopic distance and $m$-topological complexity}
In \Cref{sec: mhomotopic distance}, we revisit the notion of $m$-homotopic distance, which was originally introduced by Mac\'ias-Virg\'os, Mosquera-Lois  and Oprea in \cite{m-hom-dist}. 
The \emph{$m$-homotopic distance} between two maps $f,g \colon X \to Y$, denoted by $\Dm(f,g)$, is the smallest integer $k$ such that $X$ admits an open cover by $k+1$ sets on each of which the restrictions of $f$ and $g$ become homotopic when precomposed with any map from an $m$-dimensional CW complex.
We show in \Cref{thm: m-hom-dist(f:g) = m-secat(Pi_1)} that the $m$-homotopic distance between maps $f,g \colon X\to Y$ can be expressed as the $m$-sectional category of a fibration which is a pullback of $\pi_Y\colon PY\to Y\times Y$ along the map $(f,g)\colon X\to Y\times Y$.

In \Cref{subsec: properties of m homotopic distance} and \Cref{subsec: invariance}, we explore several new properties and establish new results and relationships with the $m$-category,
which were not considered in \cite{m-hom-dist}. 
We also obtain a cohomological lower bound on the $m$-homotopic distance.
In particular, for maps $f,g\colon X\to Y$, we show in \Cref{prop: cohomological lower bound on Dm} that, if there exist cohomology classes $u_0,\ldots,u_k \in H^*(Y \times Y;R)$ such that for each $i$, $\Delta_Y^*(u_i)=0$ and $(f,g)^*(u_0\smile \dots \smile u_k)\neq 0$, and if $m=\max\{\deg(u_i)\}$, then 
$$
    \Dm(f,g) \ge k+1.
$$

Furthermore, in \Cref{subsubsec: m cohomological distance} we introduce the notion of $m$-cohomological distance, generalizing the cohomological distance recently introduced in \cite{Coh-dist}, and show that it provides a lower bound for the $m$-homotopic distance.
The \emph{$m$-cohomological distance} between maps $f,g \colon X \to Y$ with coefficients in $R$, denoted by $\HDm(f,g;R)$, is the smallest integer $k$ such that $X$ admits an open cover by $k+1$ sets on which the induced maps in cohomology agree after precomposition with any map from an $m$-dimensional CW complex. We show the monotonicity property in \Cref{prop: monotonicity of m-cohom-dist}.
We establish the cohomological lower bound on $\HDm(f,g;R)$ in \Cref{thm : lower bound for H*Dm}. 
In particular, let $\J(f,g;R)$ denote the image of $f^*-g^* \colon H^*(Y;R)\to H^*(X;R)$. We show that, if there exist classes $u_0,\ldots,u_k \in \J(f,g;R)$ with nonzero cup product and $m=\max\{\deg(u_i)\}$, then
\[
    \HDm(f,g;R) \ge k+1.
\]
We also obtain, in \Cref{prop: m-hom-dist = hom-dist if Y is ''aspherical''}, a sufficient condition on the space $Y$ under which $\Dm(f,g) = \mathrm{D}(f,g)$.

In \Cref{sec: mTC}, we define the \emph{$m$-topological complexity} of a space $X$, denoted by $\TCm(X)$, by
\[
    \TCm(X) := \sct_m(\pi_X),
\]
where $\pi_X$ is the free path space fibration. 
In \Cref{prop: m-TC(X) = m-hom-dist(pr_1:pr_2)}, we show that this definition coincides with the notion of $m$-topological complexity introduced by Mac\'ias-Virg\'os, Mosquera-Lois, and Oprea in \cite{m-hom-dist} via the $m$-homotopic distance.
The advantage of our approach is that it allows us to derive several fundamental results, including a cohomological lower bound and a product inequality. 
In particular, in \Cref{prop: cohomological lower bound of m-TC}, we prove the following: suppose $X$ has finitely generated homology and let $\mathbb{K}$ be a field. 
If there exist classes $u_0,\ldots,u_k$ in the kernel of the cup product $\smile \colon H^*(X;\mathbb{K}) \otimes_{\mathbb{K}} H^*(X;\mathbb{K}) \to H^*(X;\mathbb{K})$, whose product is nonzero, and if $m = \max\{\deg(u_i)\}$, then
\[
    \TC^m(X) \ge k+1.
\]
Then we present examples computing $\TCm$ of complex projective spaces, real projective spaces, products of spheres, and Moore spaces in \Cref{ex: m-TC of complex projective spaces}, \Cref{ex: m-TC of real projective spaces}, \Cref{ex: m-TC of product of spheres}, and \Cref{ex: m-TC of Moore spaces}, respectively.
More precisely, $\TCm(\mathbb{C}P^n) = 0$ if $m=1$ and $\TCm(\mathbb{C}P^n) = 2n$ if $m \geq 2$, and $\TCm$ of $S^{\overline{n}} = S^{n_1} \times S^{n_2} \times \dots \times S^{n_k}$, where $n_1 < n_2 < \dots < n_k$, is given by
$$
\TCm(S^{\overline{n}}) = 
\begin{cases}
    0, & \text{ if } m < n_1,\\ 
    i+l(i), & \text{ if } n_i \leq m < n_{i+1},\\
    k+l(k), & \text{ if } m \geq n_k,
\end{cases}
$$
where $l(i)$ denotes the number of even-dimensional spheres in the set $\{S^{n_1},S^{n_2},\dots,S^{n_{i}}\}$.
We also obtain, in \Cref{prop: m-TC = TC if F is ''aspherical''}, a sufficient condition on the space $X$ under which $\TC^m(X) = \TC(X)$.
As an application of this result and \Cref{prop: m-cat = cat if X is ''aspherical''}, we compute $\mcat$ and $\TCm$ of closed orientable surfaces of genus at least 1 and non-orientable surfaces of genus strictly greater than $1$ in \Cref{ex: m-TC of surfaces}.

In \Cref{sec: H-spaces}, we study the $m$-homotopic distance between maps from or into $H$-spaces. 
For a path-connected $H$-space $X$ and maps $f,g \colon X \times X \to X$, we prove in \Cref{thm: H-space m-hom-dist(f:g) leq m-cat(X)} the inequality
\[
    \Dm(f,g) \le \ct_m(X).
\]
As a consequence, we obtain an $m$-category analogue of a theorem of Lupton and Scherer. 
In particular, for an $H$-space $X$, we show in \Cref{cor: TCm leq catm for H spaces} that
\[
    \TCm(X) = \ct_m(X).
\]
We conclude the section with an example computing the $m$-homotopic distance between the identity of the Lie group $U(2)$ and the inversion map.

\subsection{Theorems of Bochner type}
In \Cref{sec: Bochner type}, we establish several results of Bochner type for $m$-sectinal category and $m$-category.
Suppose that $p\colon E\to B$ is a fibration, where $B$ is a compact manifold with non-negative Ricci curvature and infinite fundamental group. 
In \Cref{thm: secat geq rank}, we prove the following inequality:
\[
    \sct_1(p)\geq
        \operatorname{rank}\ker\bigl(p^*\colon H^1(B;\mathbb{Q})
        \to H^1(E;\mathbb{Q})\bigr).
\]
This inequality generalizes the corresponding result for $\ct_1$ established by Oprea and Strom in \cite{MR2879375}.
In \Cref{thm: betti num leq m-cat - m-cup}, we prove an $m$-category analogue of Oprea's result \cite[Theorem 3.3]{MR1866039}. 
More precisely, suppose that $B$ is a compact manifold with non-negative Ricci curvature and infinite fundamental group. 
If $\widetilde{B} \cong T^k\times N$ is a Cheeger--Gromoll splitting of $B$, then, for every $m$ and every commutative ring $R$,
\[
    b_1(B)\leq \mcat(B)-\cl_R^m(N),
\]
where $\cl_R^m(N)$ is the $m$-cup length of $N$.
Moreover, the following three conditions are equivalent:
\begin{enumerate}
    \item $b_1(B)=\mcat(B)$ for every $m$;
    \item $b_1(B)=\mcat(B)$ for some $m\geq \dim(N)$;
    \item $B$ is a torus.
\end{enumerate}

\subsection{Notations and Conventions}

Suppose $X$ and $Y$ are topological spaces.
\begin{itemize}
\item If $U \subseteq X$, then $\iota_U \colon U \hookrightarrow X$ denotes the inclusion map into $X$.

\item If $x_0 \in X$, then $c_{x_0} \colon Y \to X$ denotes the constant map from the space $Y$ taking the value $x_0$ in $X$.

\item The maps $\pr_i \colon X \times X \to X$, $i = 1,2$, denote the projection onto the $i$th coordinate.

\item If $H \colon Y \times I \to X$ is a homotopy, then $H_t \colon Y \to X$ denotes the map $y \mapsto H(y,t)$ for each $t \in I$.
\end{itemize}

\section{$m$-sectional category}\label{sec: msecat}

In this section, we investigate the notion of the $m$-sectional category of a fibration. We show that the $m$-sectional category is a fibre-preserving homotopy invariant of a fibration. We study its relationship with the classical sectional category, establish a cohomological lower bound, and prove a product inequality.

We define this notion in a slightly different way from the original definition given by Schwarz in \cite{schwarz1961genus}. This alternative formulation will be used throughout the paper, as it is more convenient for our purposes. We also show that the two definitions are equivalent.

\begin{definition}\label{def: secatm}
Given a fibration $p \colon E \to B$, we define the \emph{$m$-sectional category} of $p$, denoted by $\sctm(p)$, as the least integer $k$ such that $B$  admits an open cover $\{U_i\}_{i=0}^{k}$, where each $U_i$ has the following property:
for any $m$-dimensional CW complex $P$ and any map $\phi \colon P \to U_i$, 
there exists a map $s \colon P \to E$ such that $p \circ s = \iota_{U_i} \circ \phi$, where $\iota_{U_i} \colon U_i \hookrightarrow B$ is the inclusion map.
This can be depicted as follows:
\[
\begin{tikzcd}
                                                                  &                                    & E \arrow[d, "p"] \\
P \arrow[r, "\phi"] \arrow[rru, "\exists ~ s", dashed, bend left] & U_i \arrow[r, "\iota_{U_i}", hook] & {B\,.}          
\end{tikzcd}
\]
If no such open cover exists, we set $\sctm(p) := \infty$.
\end{definition}

For notational simplicity and the convenience of the reader, we record the property appearing in \Cref{def: secatm} in the following form. 
This formulation will be used repeatedly in the subsequent results, and we will refer to this property instead of restating it each time.

\begin{property}
Suppose $p \colon E \to B$ is a fibration.
An open subset $U$ of $B$ will be said to satisfy property $A_{m,\,p}$ if, for every $m$-dimensional CW complex $P$ and every map $\phi \colon P \to U$, there exists a map $s \colon P \to E$ satisfying $p \circ s = \iota_U \circ \phi$.
An open cover $\mathcal{U}$ of $B$ will be said to satisfy property $A_{m,\,p}$ if every element of $\mathcal{U}$ satisfies property $A_{m,\,p}$.
\end{property}

Hence, $\sctm(p)$ is the least integer $k$ such that there exists an open cover $\{U_i\}_{i=0}^{k}$ of $B$ that satisfies property $A_{m,\,p}$.

We now state Schwarz's original definition of the $m$-sectional category, introduced under the name $m$-dimensional genus, and prove its equivalence with \Cref{def: secatm}.

\begin{definition}[{\cite[Definition 10]{schwarz1961genus}}]
The \emph{$m$-dimensional genus} of a fibration $p \colon E \to B$, denoted by $\mathrm{genus}_m(p)$, is the least integer $k$ such that $B$ admits an open cover $\{U_i\}_{i=0}^{k}$, where each $U_i$ has the following property: for any $m$-dimensional CW complex $P$ and any map $\phi \colon P \to U_i$, the pullback of fibration $p$ along $\iota_{U_i} \circ \phi$ has a global section, where $\iota_{U_i} \colon U_i \to B$ is the inclusion map.
This can be depicted as follows:
\[
\begin{tikzcd}
E' \arrow[d, "p'"] \arrow[rr, "f"]    \arrow[drr, phantom, "\square"]     &                               & E \arrow[d, "p"] \\
P \arrow[r, "\phi"] \arrow[u, "s'", dashed, bend left=49] & U_i \arrow[r, "\iota_{U_i}", hook] & {B\,,}          
\end{tikzcd}
\]
where $p' \colon E' \to B$ is the pullback of $p$ along $\iota_{U_i} \circ \phi$.
If no such open cover exists, we set $\mathrm{genus}_m(p) := \infty$.
\end{definition}

\begin{property}
Suppose $p \colon E \to B$ is a fibration.
An open subset $U$ of $B$ will be said to satisfy property $A_{m,\,p}'$ if, for every $m$-dimensional CW complex $P$ and every map $\phi \colon P \to U$, the pullback of $p$ along $\iota_U \circ \phi$ has a global section. 
An open cover $\mathcal{U}$ of $B$ will be said to satisfy property $A_{m,\,p}'$ if every element of $\mathcal{U}$ satisfies property $A_{m,\,p}'$.
\end{property}

\begin{proposition}
Suppose $p \colon E \to B$ is a fibration. 
Then an open subset of $B$ satisfies property $A_{m,\,p}$ if and only if it satisfies the property $A_{m,\,p}'$. 
In particular, 
$$
    \sctm(p) = \mathrm{genus}_m(p).
$$ 
\end{proposition}

\begin{proof}
Let $U$ be an open subset of $B$, and $P$ be an $m$-dimensional CW complex with a map $\phi \colon P \to U$. 
Suppose $p' \colon E' \to P$ is the pullback of $p$ along $\iota_U \circ \phi$.

If $U$ satisfies property $A_{m,\,p}'$, then there exists a global section $s' \colon P \to E'$ of $p'$.
\[
\begin{tikzcd}
E' \arrow[d, "p'"] \arrow[rr, "f"]                                           &                               & E \arrow[d, "p"] \\
P \arrow[r, "\phi"'] \arrow[u, "s'", bend left=49] \arrow[rru, "f \circ s'"] & U \arrow[r, "\iota_U"', hook] & B               
\end{tikzcd}
\]
Then $s:= f \circ s'$ is a map such that $p \circ s = \iota_U \circ \phi$, since
$$
    p \circ s 
        = p \circ f \circ s' 
        = \iota_U \circ \phi \circ p' \circ s' 
        = \iota_U \circ \phi.
$$

Conversely, suppose $U$ satisfies property $A_{m,\,p}$. 
Then there exists a map $s \colon P \to E$ such that $p \circ s = \iota_U \circ \phi$. 
Hence, by the universal property of pullbacks, there exists a unique map $s \colon P \to E'$ such that the following diagram commutes.
\[
\begin{tikzcd}
P \arrow[rdd, "\mathrm{Id}_P"', bend right] \arrow[rrrd, "s", bend left] \arrow[rd, "{\exists!\, s'}", dotted] &                                   &                                &                  \\
        & E' \arrow[d, "p'"] \arrow[rr, "f"] \arrow[drr, phantom, "\square"]  
        &   & E \arrow[d, "p"] \\
        & P \arrow[r, "\phi"]   & U \arrow[r, "\iota_{U}", hook] & B.               
\end{tikzcd}
\]
Then $p' \circ s' = \mathrm{Id}_P$ implies $U$ satisfies property $A_{m,\,p}'$.
\end{proof}

Next, we prove some basic homotopy theoretic properties of $m$-sectional category.
We first recall the notion of fibre-preserving homotopy equivalence in the sense of Grant, see \cite[p.~292]{Grant-ParaTC-group}.

Let $p \colon E \to B$ and $p' \colon E' \to B'$ be fibrations.
A \emph{fibre-preserving map} from $p$ to $p'$ is a pair of maps $f \colon E \to E'$ and $g \colon B \to B'$ satisfying $p' \circ f = g \circ p$, and is denoted by $(f,g) \colon p \to p'$.
A \emph{fibre-preserving homotopy} is a pair of maps $F \colon E \times I \to E'$ and $G \colon B \times I \to B'$
such that $p' \circ F = G \circ (p \times \mathrm{Id}_I)$, and is denoted by $(F,G)$.
The fibrations $p$ and $p'$ are said to be \emph{fibre-preserving homotopy equivalent} if there exist fibre-preserving maps $(f,g) \colon p \to p'$ and $(f',g') \colon p' \to p$ such that $(f \circ f',\, g \circ g')$ is fibre-preserving homotopic to $(\mathrm{Id}_{E'},\, \mathrm{Id}_{B'})$ and $(f' \circ f,\, g' \circ g)$ is fibre-preserving homotopic to $(\mathrm{Id}_{E},\, \mathrm{Id}_{B})$.

We now establish the fibre-preserving homotopy invariance of the $m$-sectional category.
\begin{proposition}\label{prop: fibre-preserving homotopy invariance}
Suppose $p \colon E \to B$ and $p' \colon E' \to B'$ are fibrations such that the following diagram commutes up to homotopy
\begin{equation}
\label{diag: pullback digram}
\begin{tikzcd}
E' \arrow[d, "p'"'] \arrow[r, "f'"] & E \arrow[d, "p"] \\
B' \arrow[r, "f"]                   & {B \,.}         
\end{tikzcd}
\end{equation}
\begin{enumerate}
\item If $f$ admits a right homotopy inverse, then 
$
        \sctm(p) \leq \sctm(p').
$

\item If $f$ is a homotopy equivalence and $f'$ admits a right homotopy inverse, then 
$
    \sctm(p) = \sctm(p').    
$
\end{enumerate}
In particular, $\sctm$ is a fibre-preserving homotopy invariant of a fibration.
\end{proposition}

\begin{proof}
(1) Suppose $g \colon B \to B'$ is a right homotopy inverse of $f$, i.e., $f \circ g \simeq \id_{B}$.
Suppose $U$ is an open subset of $B'$ that satisfies property $A_{m,\,p'}$.
We claim that $V = g^{-1}(U)$ is an open subset of $B$ that satisfies property $A_{m,\,p}$.

Let $P$ be an $m$-dimensional CW complex with a map $\psi \colon P \to V$. 
We want to produce a map $s \colon P \to E$ such that $p \circ s = \iota_{V} \circ \psi$.
Define $\widetilde{\phi} := g \circ \iota_{V} \circ \psi$. 
Then $\widetilde{\phi}(P) \subset U$.
If $\phi \colon P \to U$ is the map induced from $\widetilde{\phi}$ by restricting the codomain, then there exists a map $s' \colon P \to E'$ such that $p' \circ s' = \iota_{U} \circ \phi$.
Let $r := f' \circ s'$.
Hence, we have a diagram
\[
\begin{tikzcd}
                                                  &                                & E' \arrow[r, "f'", shift left] \arrow[d, "p'"] & E \arrow[d, "p"]             &                                  &                                                      \\
P \arrow[r, "\phi"] \arrow[rru, "s'", shift left] & U \arrow[r, "\iota_{U}", hook] & B' \arrow[r, "f", shift left]                  & B \arrow[l, "g", shift left] & V \arrow[l, "\iota_{V}"', hook'] & P. \arrow[l, "\psi"'] \arrow[llu, "r"', shift right]
\end{tikzcd}
\]
Then
$
p \circ r
    = p \circ f' \circ s'
    \simeq f \circ p' \circ s '
    = f \circ \iota_{U} \circ \phi 
    = f \circ \widetilde{\phi}
    = f \circ g \circ \iota_{V} \circ \psi 
    \simeq \iota_{V} \circ \psi.
$
Let $H_t \colon P \to B$ be a homotopy such that $H_0= p \circ r$ and $H_1 = \iota_{V} \circ \psi$.
Since $p$ is a fibration, there exists a homotopy $\widetilde{H}_t \colon P \to E$ such that $p \circ \widetilde{H}_t = H_t$ and $\widetilde{H}_0 = r$. 
Thus, $s := \widetilde{H}_1$ is the required map, since $p \circ \widetilde{H}_1 = H_1 = \iota_{V} \circ \psi$.

(2) Suppose $g' \colon E \to E'$ is a right homotopy inverse of $f'$, and $g \colon B \to B'$ is a homotopy inverse of $f$. 
Then $p \circ f' \simeq f \circ p'$ implies
$$
    g \circ p
        \simeq g \circ (p \circ f') \circ g' 
        \simeq g \circ (f \circ p') \circ g'
        \simeq p' \circ g'.
$$
Applying (1) to the homotopy equivalence above, we get $\sctm(p') \leq \sctm(p)$. 
\end{proof}

\begin{proposition}
\label{prop: m-secat under pullback}
Suppose $p \colon E \to B$ is a fibration such that the diagram \eqref{diag: pullback digram} is a pullback.
Then $\sctm(p') \leq \sctm(p)$. 
\end{proposition}

\begin{proof}
Suppose $U$ is an open subset of $B$ that satisfies property $A_{m,\,p}$.
Let $V = f^{-1}(U)$, and let $\widetilde{f} \colon V \to U$ be the restriction of $f$.
Suppose $\phi \colon P \to V$ is a map, where $P$ is an $m$-dimensional CW complex.
Then there exists a map $s \colon P \to E$ such that $p \circ s = \iota_U \circ (\widetilde{f} \circ \phi)$.
As $\iota_U \circ \widetilde{f} = f \circ \iota_V$, we get
$p \circ s = f \circ \iota_V \circ \phi$.
Thus, by the universal property of pullbacks, there exists a unique map $s' \colon P \to E'$ such that the following diagram
\begin{equation}
\label{diag: universal diag for pullback}
\begin{tikzcd}
&       & E' \arrow[d, "p'"'] \arrow[r, "f'"] & E \arrow[d, "p"] \\
P \arrow[r, "\phi"] \arrow[rrru, "s", bend left=49] \arrow[rru, "s'", dotted, shift left] & V \arrow[r, "\iota_V", hook] & B' \arrow[r, "f"]                   & B               
\end{tikzcd}
\end{equation}
is commutative.
Hence, $p' \circ s' = \iota_V \circ \phi$, i.e., $V$ satisfies the property $A_{m,\,p'}$.
\end{proof} 

The following result establishes the monotonicity property of the $m$-sectional category.
\begin{proposition}
\label{prop: monotonicity of m-secat}
Suppose $p \colon E \to B$ is a fibration. 
Then
\begin{enumerate}
\item an open subset of $B$ satisfies property $A_{m,\,p}$ if and only if it satisfies the property $A_{n,p}$ for all $n \leq m$.
In particular, 
$
    \sct_n(p) \leq \sctm(p)
$ 
for all $n \leq m$.

\item $\sctm(p) \leq \sct(p)$ for all $m$.
\end{enumerate}
\end{proposition}

\begin{proof}
(1) Suppose $U$ is an open subset of $B$ that satisfies property $A_{m,\,p}$.
Let $\psi \colon Q \to U$ be a map, where $Q$ is a $n$-dimensional CW complex and $n \leq m$.
If $q_0 \in Q$ is a point, then
$$
    P := Q \vee S^m = \frac{Q \coprod D^m}{q_0 \sim z},  \quad \text{for all }z \in S^{m-1}
$$
is an $m$-dimensional CW complex, and $\psi$ extends to a map $\phi \colon P \to U$, given by $\phi([q]) = \psi(q)$ and $\phi([x]) = \psi(q_0)$ for $q \in Q$ and $x \in D^m$.
Thus, there exists a map $s \colon P \to E$ such that $p \circ s = \iota_{U} \circ \phi$.
Hence, $\left.s\right|_{Q} \colon Q \to E$ is the map such that $p \circ \left.s\right|_{Q} = \iota_{U} \circ \left.\phi\right|_{Q} = \iota_{U} \circ \psi$.

(2) Suppose $U$ is an open subset of $B$ with a section $s$ of $p$ over $U$.
Suppose $P$ is an $m$-dimensional CW complex with a map $\phi \colon P \to U$.
Then we have the following commutative diagram
\[
\begin{tikzcd}
                                                           &                                              & E \arrow[d, "p"] \\
P \arrow[r, "\phi"] \arrow[rru, "s \circ \phi", bend left] & U \arrow[r, "\iota_U", hook] \arrow[ru, "s"] & {B \, .}        
\end{tikzcd}
\]
Hence, the map $s' := s \circ \phi$ satisfies $p \circ s' = \iota_U\circ \phi$.
\end{proof}

\subsection{Bounds on $m$-sectional category}
\hfill\\ \vspace{-0.7em}

We establish a cohomological lower bound for the $m$-sectional category of a fibration. 
This bound is obtained by adapting classical cohomological techniques to the $m$-dimensional setting.
\begin{proposition}
\label{prop: cohomological lower bound of m-secat}
Suppose $p \colon E \to B$ is a fibration. 
If there exist cohomology classes $u_0, u_1,\ldots, u_k \in H^*(B;R)$ (for any commutative ring $R$) with 
$$
    p^*(u_0) = p^{*}(u_1) = \cdots = p^{*}(u_k) = 0
        \quad \text{ and } \quad
    u_0 \smile u_1 \smile \cdots \smile u_k \neq 0,
$$
then 
$
    \sctm(p) \geq k+1,
$
where $m = \max\{\deg(u_i) \mid i=0,1,\ldots,k\}$.
\end{proposition}

\begin{proof}
Suppose that $\sctm(p)\leq k$. 
Let $\{U_0,U_1,\dots,U_k\}$ be an open cover of $B$ that satisfies property $A_{m,\,p}$. 

We claim that $\iota_0^*(u_0) = \iota_1^*(u_1) = \cdots = \iota_k^*(u_k) = 0$, where $\iota_i \colon U_i \hookrightarrow B$ denotes the inclusion map. 
Suppose $\iota_i^{*}(u_i) \neq 0$ for some $i$.
By the cellular approximation theorem, there exists a CW complex $W_i$ and a weak homotopy equivalence $h_i \colon W_i \to U_i$.
Then $h_i^{*}(\iota_i^*(u_i)) \neq 0$. 

\noindent
Case (1): Suppose $\dim(W_i) \leq m$. 
Then, by \Cref{prop: monotonicity of m-secat} (1), there exists a map $s_i \colon W_i \to E$ such that $p \circ s_i = \iota_i \circ h_i$.
Hence, $h_i^*(\iota_i^{*}(u_i)) = s_i^{*}(p^*(u_i)) = s_i^{*}(0) = 0$ gives a contradiction.

\noindent
Case (2): Suppose $\dim(W_i) > m$.
Let $(W_{i})_m$ denote the $m$-skeleton of $W_i$ and $\iota \colon (W_i)_m \hookrightarrow W_i$ denotes the inclusion map.
Then $\iota^* \colon H^{j}(W_i;R) \to H^{j}((W_i)_m;R)$ is injective for $j \leq m$.
Hence, $m \geq \deg(u_i)$ implies $\iota^*(h_i^{*}(\iota_i^*(u_i))) \neq 0$. 
Using property $A_{m,\,p}$, there exists a map $t_i \colon (W_i)_m \to E$ such that $p \circ t_i = \iota_i \circ h \circ \iota$.  
Hence, $\iota^*(h_i^*(\iota_i^{*}(u_i))) = t_i^{*}(p^*(u_i)) = t_i^{*}(0) = 0$ gives a contradiction.

Using the long exact sequence in cohomology for the pair $(B,U_i)$, we see that there exists a relative cohomology class $v_i \in H^*(B, U_i; R)$ such that $j_i^*(v_i) = u_i$, where $j_i \colon B \hookrightarrow (B,U_i)$ is the inclusion map.
Thus, we obtain
$$
    v_0 \smile v_1 \smile \cdots \smile v_k 
        \in H^*(B, \cup^k_{i=0} U_i; R) = H^*(B,B; R)= 0.
$$
Moreover, by the naturality of the cup product, we have 
$$
    u_0 \smile u_1 \smile \cdots \smile u_k 
        = j^*(v_0 \smile v_1 \smile \cdots \smile v_k) = 0,
$$ 
where $j \colon B \hookrightarrow (B,B)$ is the inclusion map.
This contradicts our assumption.
\end{proof}

Next, we state two important results from \cite{schwarz1961genus}, which relate $\sct_m$ to the classical sectional category. 
These results will be used in the upcoming sections to establish relationships between various induced $m$-dimensional invariants.

\begin{proposition}[{\cite[Proposition 17 and Lemma 3]{schwarz1961genus}}]
\label{prop: m-secat = secat if F is ''apsherical''}
Suppose $F \hookrightarrow E \xrightarrow{p} B$ is a fibration, where $B$ is a CW complex. 
If $\pi_k(F)=0$ for all $k \geq m$, then
$$
    \sctm(p) = \sct(p).
$$
\end{proposition}

Recall that $\phi_i\colon G_i\to H_i,~i=1,2$ are group homomorphisms. We say that $\phi_1$ and $\phi_2$ are equivalent as homomorphisms if there exist group isomorphisms $\alpha\colon G_1\to G_2$ and $\beta\colon H_1\to H_2$ such that $\phi_2\circ \alpha = \beta\circ \phi_1$.

\begin{corollary}
Suppose $\phi\colon G\to H$ is a group homomorphism, where $H$ is an abelian group. For $n\geq 1$, there exists a fibration $p\colon E\to K(H,n)$ such that $p_*\colon \pi_n(E)\to H$ and $\phi$ are equivalent as homomorphisms. Then
$$\sctm(p) = 
\begin{cases}
 0 &\text{ for all } m< n,\\
 \sct(p) &\text{ for all } m\geq n+1,\\
 \sct(p) &\text{ for all } m\geq n, \text{ provided } \phi \text{ is a monomorphism}.
\end{cases}
$$
\end{corollary}

\begin{proof}
By universal coefficient theorem for cohomology together with \cite[Theorem 4.57]{hatcher}, there exits a map $f\colon K(G,n)\to K(H,n)$ such that $f_* = \phi$. Replacing $f$ by its mapping-path fibration, we obtain a fibration $p\colon E\to K(H,n)$, where $E$ is homotopy equivalent to $K(G,n)$. Let $F$ be a fibre of the fibration $p$. The long exact sequence of homotopy groups associated to the fibration $p$ yields $\pi_i(F) = 0$ for all $i\geq n+1$. By Proposition \Cref{prop: m-secat = secat if F is ''apsherical''}, we have $\sctm(p) = \sct(p)$ for all $m\geq n+1$. In addition, if $\phi$ is a monomorphism, then $\pi_i(F) = 0$ for all $i\geq n$. Thus, $\sct_n(p) = \sct(p)$. Finally, by cellular approximation, we obtain $\sctm(p) = 0$ for all $m<n$. This completes the proof. 
\end{proof}



\begin{proposition}[{\cite[Proposition 19]{schwarz1961genus}}]
\label{prop: secat leq m-secat + [k/m+1]}
Suppose $p \colon E \to B$ is a fibration, where $B$ is a $k$-dimensional CW complex.
Then 
$$
    \sct(p) \leq \sctm(p) + \left[\frac{k}{m+1}\right], 
        \quad \text{and} \quad 
    \sct(p) \leq \max\{\sct_{k-1}(p),2\},
$$
where $[w]$ denotes the greatest integer less than or equal to $w$.
\end{proposition}

Before stating the corollary of this result, we recall the notion of the homotopy dimension of a CW complex.

\begin{definition}
Suppose $X$ is a CW complex. 
The \emph{homotopy dimension} of $X$, denoted $\hdim(X)$, is defined to be
$$
    \hdim(X) 
        := \min\{\dim(X') \mid X' \text{ is a CW complex such that } X'\simeq X\}.
$$
\end{definition}

\begin{corollary}
\label{cor: secat leq m-secat + [hdim/m+1]}
Suppose $p \colon E \to B$ is a fibration, where $B$ is a finite-dimensional CW complex.
Then 
$$
    \sct(p) \leq \sctm(p) + \left[\frac{\hdim(B)}{m+1}\right].
$$
where $[w]$ denotes the greatest integer less than or equal to $w$.   
In particular, 
$$
    \sctm(p) = \sct(p)
$$ 
for all $m\geq \hdim(B).$
\end{corollary}

\begin{proof}
Let $B'$ be a finite-dimensional CW complex and $f \colon B' \to B$ be a homotopy equivalence.
Suppose $p' \colon E' \to B'$ is the pullback of $p$ along $f$, see diagram \eqref{diag: pullback digram}.
Hence, by \Cref{prop: secat leq m-secat + [k/m+1]}, it follows that 
$$
    \sct(p') \leq \sctm(p') + \left[\frac{\dim(B')}{m+1}\right].
$$
Since $\sctm(p') = \sctm(p)$, by \Cref{prop: fibre-preserving homotopy invariance} (1) and \Cref{prop: m-secat under pullback}, and $\sct(p) = \sct(p')$, the desired result follows.
\end{proof}

\subsection{Normal spaces and product inequality}
\hfill\\ \vspace{-0.7em}

In this subsection, we establish the product inequality for the $m$-sectional category. 
A key step in the proof is to verify that Property $A_{m,p}$ is preserved under taking open subsets and disjoint unions, which we prove below.

\begin{proposition}
\label{prop: inheritance of m-secat}
Property $A_{m,\,p}$ is inherited by open subsets and finite disjoint unions.
\end{proposition}

\begin{proof}
Let $p \colon E \to B$ be a fibration.
Suppose $U$ is an open subset of $B$ that satisfies property $A_{m,\,p}$, and let $U' \subseteq U$ be an open subset. 
Let $\phi \colon P \to U'$ be a map, where $P$ is an $m$-dimensional CW complex. 
Let $j \colon U' \hookrightarrow U$ be the inclusion map. 
Then there exists a map $s \colon P \to E$ such that $p \circ s = \iota_U \circ (j \circ \phi)$. 
As $\iota_U \circ j = \iota_{U'}$, it follows that $p\circ s = \iota_{U'} \circ \phi$. 

Suppose $U_1$ and $U_2$ are disjoint open subsets of $B$ that satisfies the property $A_{m,\,p}$.
Let $\phi \colon P \to U_1 \sqcup U_2$ be a map, where $P$ is an $m$-dimensional CW complex. 
Let $P_{i}$ be the CW complex defined as $P_{i} = \phi^{-1}(U_i)$.
Note that $P_{i}$ consists of connected components $C$ of $P$ such that $\phi(C) \subseteq U_i$.
Let $\phi_i \colon P_i \to U_i$ denote the map induced from $\phi$ by restricting its domain and codomain.
As $\dim(P_{i}) \leq m$, it follows by \Cref{prop: monotonicity of m-secat} that there exists a map $s_i \colon P_i \to E$ such that $p \circ s_i = \iota_{U_i} \circ \phi_i$.
As $P = P_1 \sqcup P_2$, we get a map $s \colon P \to E$ given by 
$$
    s(x) = 
\begin{cases}
    s_1(x), & \text{for } x \in P_1,\\
    s_2(x), & \text{for } x \in P_2.
\end{cases}
$$
Clearly, $s$ satisfies $p \circ s = \iota_{U_1 \sqcup U_2} \circ \phi$.
\end{proof}

In order to prove a product inequality for $\sctm$, we revisit the following results. 
Recall that an open cover $\mathcal{U}=\{U_0, \dots, U_r\}$ of $X$ is called an \emph{$(n+1)$-cover} if every subcollection $\{U_{j_0}, \dots, U_{j_n}\}$ of $(n+1)$ sets from $\mathcal{U}$ also covers $X$.

\begin{lemma}[{\cite{F-G-L-O}}\label{lem: equivalent criterion n+1 cover}]
A cover $\mathcal{V}=\{V_0,\dots, V_{k+n}\}$ of $X$ is $(n+1)$-cover if and only if each $x\in X$ is contained in at least $k+1$ sets of $\mathcal{V}$.    
\end{lemma}

\begin{theorem}[{\cite[Theorem 2.5]{dranishnikov2009lusternik}}]
\label{thm: extend cover to k+1 cover}
Let $\mathcal{U}=\{U_0, \dots, U_k\}$ be an open cover for a normal topological space $X$. 
Then, for any $l \geq k$, there is an open $(k+1$)-cover $\{U_0, \dots, U_l\}$ for $X$, extending $\mathcal{U}$ such that for $n>k$, $U_n$ is a disjoint union of open sets that are subsets of the $U_j$, where $0 \leq j \leq k$.   
\end{theorem}

\begin{proposition}
\label{prop: prod-ineq for m-secat}
Suppose $p_i \colon E_i \to B_i$ (for $i=1,2$) are fibrations, where $B_1$ and $B_2$ are normal spaces. 
Then 
$$
    \sctm(p_1 \times p_2) \leq \sctm(p_1) + \sctm(p_2),
$$
where $p_1 \times p_2 \colon E_1 \times E_2 \to B_1 \times B_2$ denotes the product fibration.
\end{proposition}

\begin{proof}
    Assume that $\sctm(p_1) = k$ and $\sctm(p_2) = n$. 
    Suppose $\mathcal{U} = \{U_0, \dots, U_{k}\}$ and $\mathcal{V} = \{V_0, \dots, V_{n}\}$ are open covers of $B_1$ and $B_2$ that satisfy properties $A_{m,p_1}$ and $A_{m,p_2}$, respectively. 
    By \cref{thm: extend cover to k+1 cover}, the open cover $\mathcal{U}$ can be extended to a $(k+1)$-cover $\mathcal{U}' = \{U_0,\dots, U_{k+n}\}$ such that for $N > k$, $U_N$ is a disjoint union of open sets that are subsets of the $U_j$ for $0 \leq j \leq k$.
    Similarly, the open cover $\mathcal{V}$ can be extended to an $(n+1)$-cover $\mathcal{V}'=\{V_0,\dots, V_{k+n}\}$ such that for $N > n$, $V_N$ is a disjoint union of open sets that are subsets of the $V_j$ for $0 \leq j \leq n$. 
    Consider the collection 
    $$
        \mathcal{W}
            := \{W_i := U_i \times V_i \mid U_i \in \mathcal{U}', V_i \in \mathcal{V}', 0\leq i\leq n+k\}.
    $$ 
    
    We now show that $\mathcal{W}$ covers $B_1\times B_2$. 
    Suppose, on the contrary, there exists $(b_1,b_2)\in B_1\times B_2$ such that $(b_1,b_2)\notin W_i$ for all $0\leq i\leq n+k$. 
    Since $\mathcal{U}'$ is a $(k+1)$-cover, by \Cref{lem: equivalent criterion n+1 cover} $b_1$ belongs to at least $n+1$ sets of $\mathcal{U}'$. 
    Without loss of generality, denote these sets by $U_0,\dots, U_n$. 
    Then, by our assumption, $b_2 \notin V_0\cup \dots \cup V_n$. 
    This implies that $b_2$ can lie only in $V_{n+1}, \dots, V_{n+k}$. 
    But this is a contradiction, as $\mathcal{V}'$ is an $(n+1)$-cover and $b_2$ must lie in at least $k+1$ sets of $\mathcal{V}'$.

    We will now show that $\sctm(p_1\times p_2)\leq k+n$. 
    Let $\phi \colon P \to W_i$ be a map from an $m$-dimensional CW complex $P$ to $W_i=U_i \times V_i \in \mathcal{W}$ for some $i$. 
    Then $\mathrm{pr}_1 \circ \phi$ and $\mathrm{pr}_2 \circ \phi$ are maps from $P$ to $U_i$ and $V_i$, respectively, where $\mathrm{pr}_1 \colon W_i \to U_i$ and $\mathrm{pr}_2 \colon W_i \to V_i$ denote the projection maps. 
    Using \Cref{prop: inheritance of m-secat}, we obtain maps $s_1 \colon P \to E_1$ and $s_2 \colon P \to E_2$ such that $p_1 \circ s_1 = \iota_{U_i} \circ (\mathrm{pr}_1 \circ \phi)$ and $p_2 \circ s_2 = \iota_{V_i} \circ (\mathrm{pr}_2 \circ \phi)$, where $\iota_{U_i} \colon U_i \hookrightarrow B_1$ and $\iota_{V_i} \colon V_i \hookrightarrow B_2$ denote the inclusion maps. 
    Then the map $s \colon P \to E_1 \times E_2$, defined by $s(p) = (s_1(p),s_2(p))$, satisfies $(p_1\times p_2)\circ s = \iota_{W_i} \circ \phi$, where $\iota_{W_i} \colon W_i \hookrightarrow B_1 \times B_2$ is the inclusion map.
    Thus, $\sctm(p_1 \times p_2) \leq k+n$. 
\end{proof}

\section{$m$-category}
\label{sec: m-category}

In this section, we study the notion of the $m$-dimensional Lusternik--Schnirelmann category and its connections with the classical LS category. 
We prove the monotonicity property and define the notion of the $m$-category of maps. 
We explore its relationship with the $m$-sectional category, derive a cohomological lower bound for the $m$-category, and establish a product inequality.

\begin{definition}[{\cite{Fox-LS-category}, \cite[Definition 20]{schwarz1961genus}}]
The \emph{$m$-dimensional LS category (or $m$-category)} of a topological space $X$, denoted by $\mcat(X)$, is the least integer $k$ such that $X$ admits an open cover $\{U_i\}_{i=0}^{k}$, where each $U_i$ has the following property: for any map $\phi \colon P \to U_i$ from an $m$-dimensional CW complex $P$, the composition $\iota_{U_i} \circ \phi$ is nullhomotopic, where $\iota_{U_i} \colon U_i \hookrightarrow X$ is the inclusion map.
\end{definition}

\begin{property}
An open subset $U$ of a topological space $X$ will be said to satisfy property $B_m$ if, for every $m$-dimensional CW complex $P$ and every map $\phi \colon P \to U$, the composition $\iota_U \circ \phi$ is nullhomotopic.
An open cover $\mathcal{U}$ of $B$ will be said to satisfy property $B_m$ if every element of $\mathcal{U}$ satisfies property $B_m$.
\end{property}

\begin{proposition}
\label{prop: monotonicity of m-cat}
Let $X$ be a topological space. 
Then the following statements hold.
\begin{enumerate}
\item An open subset of $X$ satisfies property $B_m$ if and only if it satisfies the property $B_n$ for all $n \leq m$.
In particular, 
$
    \ct_n(X) \leq \mcat(X)
$ 
for all $n \leq m$.

\item $\mcat(X) \leq \ct(X)$ for all $m$.
\end{enumerate}
\end{proposition}

\begin{proof}
The proof of (1) follows similar argument as in \Cref{prop: monotonicity of m-secat}.  

(2) If $U$ is a categorical open subset of $X$, i.e., $\iota_U \colon U \hookrightarrow X$ is nullhomotopic, then $\iota_U \circ \phi$ is nullhomotopic for any map $\phi \colon P \to U$. 
Hence, the desired the inequality follows. 
\end{proof}

\begin{proposition}[{\cite[Lemma 2.1]{MR2879375}}]
If $X$ is $r$-connected, then $\ct_m(X) = 0$ for all $m \leq r$.  
\end{proposition}

\subsection{$m$-category of a map}\label{subsec: m-category of map}

\begin{definition}
The \emph{$m$-dimensional LS category} of a map $f \colon X \to Y$, denoted by $\mcat(f)$, is defined to be the least integer $k$ such that $X$ admits an open cover $\{U_i\}_{i=0}^{k}$, where each $U_i$ has the following property: for any map $\phi \colon P \to U_i$ from an $m$-dimensional CW complex P, the composition $f \circ \iota_{U_i} \circ \phi$ is nullhomotopic, where $\iota_{U_i} \colon U_{i} \hookrightarrow X$ is the inclusion map.
\end{definition}

Observe that $\mcat(\id_X) = \mcat(X)$, so the $m$-category of a map extends the notion of $m$-category of a space.

\begin{property}
Suppose $f \colon X \to Y$ is a continuous map. 
An open subset $U$ of $X$ will be said to satisfy property $C_{m,\,f}$ if, every $m$-dimensional CW complex $P$ and every map $\phi \colon P \to U$, the composition $f \circ \iota_U \circ \phi$ is nullhomotopic.
An open cover $\mathcal{U}$ of $B$ will be said to satisfy property $C_{m,\,f}$ if every element of $\mathcal{U}$ satisfies property $C_{m,\,f}$.
\end{property}

The following proposition is the $m$-category analogue of \cite[Proposition 3.3 (2)]{Fox-LS-category}.

\begin{proposition}\label{prop: secatm and catmf}
Suppose $p \colon E \to B$ is a surjective fibration such that the diagram \eqref{diag: pullback digram} is a pullback.
\begin{enumerate}
    \item Then $\sctm(p') \leq \mcat(f)$.  
    \item If $E$ is $r$-connected, then $\sctm(p') = \mcat(f)$ for all $m \leq r$.
\end{enumerate}
\end{proposition}

\begin{proof}
(1) Let $V$ be an open subset of $B'$ that satisfies property $C_{m,\,f}$. 
Suppose $\phi \colon P \to V$ is a map, where $P$ is an $m$-dimensional CW complex.
Then $f \circ \iota_V \circ \phi$ is nullhomotopic.
Consider the homotopy $H_t \colon P \to B$ such that $H_0 = f \circ \iota_V \circ \phi$ and $H_1 = c_{b_0}$, where $c_{b_0} \colon P \to B$ is the constant map which takes the value $b_0 \in B$.
Let $e_0 \in p^{-1}(b_0)$.
Since $p$ is a fibration, there exists a homotopy $\widetilde{H}_t \colon P \to E$ such that $p \circ \widetilde{H}_t = H_t$ and $\widetilde{H}_1 = c_{e_0}$.
Hence, $s := \widetilde{H}_0$ is a map such that $p \circ s = p \circ \widetilde{H}_0 = H_0 = f \circ \iota_V \circ \phi$.
Thus, by the universal property of pullbacks, there exists a unique map $s' \colon P \to E'$ such that the diagram \eqref{diag: universal diag for pullback} commutes. 
Hence, $p' \circ s' = \iota_V \circ \phi$, i.e., $V$ satisfies the property $A_{m,\,p'}$.

(2) Suppose $V$ is an open subset of $B'$ that satisfies property $A_{m,\,p'}$.
Let $P$ be the $m$-dimensional CW complex with a map $\phi \colon P \to V$. 
Then there exists a map $s' \colon P \to E'$ such that $p' \circ s' = \iota_{V} \circ \phi$. 
Hence, $p \circ f' \circ s' = f \circ p' \circ s' = f \circ \iota_V \circ \phi$.
Since the composition $f' \circ s'$ is map from an $m$-dimensional CW complex to a $r$-connected space $E$ and $m \leq r$, it follows that $f' \circ s'$ is nullhomotopic. 
Hence, $f \circ \iota_{V} \circ \phi$ is also nullhomotopic, i.e., $V$ satisfies property $C_{m,\,f}$.
\end{proof}

Taking $f$ to be the identity map in the above proposition, we obtain the following result.

\begin{corollary}
\label{cor: m-secat leq m-cat}
Let $p \colon E \to B$ be a surjective fibration. 
\begin{enumerate}
    \item Then $\sctm(p) \leq \mcat(B).$ 
    In particular, if $B$ is $r$-connected, then $\sctm(p) = 0$ for all $m \leq r$.
    \item If $E$ is $r$-connected, then $\sctm(p) = \mcat(B)$ for all $m \leq r$.
\end{enumerate}
\end{corollary}

For a topological space $X$ with a chosen basepoint $x_0 \in X$, the path space fibration
$$
    e_X \colon P_{x_0}X \to X
$$
is given by $e_X(\alpha) = \alpha(1)$, where $P_{x_0}X = \{\alpha \in PX \mid \alpha(0)= x_0\}$.

\begin{corollary}
\label{cor: m-cat = m-secat(e_X)}
If $(X,x_0)$ is a path-connected topological space, then
$$
    \mcat(X) = \sctm\left(e_X \colon P_{x_0}X \to X \right)
$$
for all $m$.
\end{corollary}

As an application of our previous results, we now provide various examples of fibrations for which $\sct_m < \sct_n$ for $m<n$, as well as examples for which $\sct_m = \sct$.

\begin{example}
\label{Example 1}
Consider the free path space fibration $\pi_{S^n} \colon P(S^n) \to S^n \times S^n$, where $n \geq 2$. 
Then 
$\sct_m(\pi_{S^n}) = 0$ for all $m \leq n-1$, by \cref{cor: m-secat leq m-cat} (1). 
However, $\sct(\pi_{S^n})$ is nonzero. 
In fact, it is equal to the topological complexity of $S^n$, which is either $1$ or $2$. 
This shows that $\sctm$ can be strictly smaller than $\sct$. 
\end{example}

\begin{example}
\label{Example 2}
Recall from \cite[Proposition 6.6 and Corollary 10.6]{Mc} that $K(G,n)$ is a topological group if $G$ is a discrete abelian group.
Consider a nontrivial principal $K(G,1)$-bundle $p \colon E \to B$, where $G$ is a discrete abelian group and $B$ is a CW complex. 
If the connectivity of $B$ is greater than $1$, then the bundle is trivial. 
Indeed, \cite[Theorem 13.1]{husemoller1994fibre} gives a bijection between the set of isomorphism classes of principal $K(G,1)$-bundles over $B$ and the cohomology group $H^2(B,G)$, which turns out to be trivial if $B$ is $2$-connected or higher, due to the Hurewicz theorem, followed by the universal coefficient theorem. 
Hence, assume that $B$ is $1$-connected. 
Then $\sct_1(p)=0$, by \Cref{cor: m-secat leq m-cat} (1). 
But \Cref{prop: m-secat = secat if F is ''apsherical''} implies that $\sct_2(p)=\sct(p)$, which is nonzero as $p$ is nontrivial. 
This shows that $\sct_1$ can be strictly smaller than $\sct_2$. 
An explicit example of such a map $p$ is given by the Hopf fibration $S^1 \hookrightarrow S^{2n+1} \xrightarrow{p}\mathbb{C}P^n$.
\end{example}

\begin{example}
\label{Example 3}
Consider the universal covering map $p \colon S^n\to \mathbb{R}P^n$ for $n \geq 2$. 
The mod $2$ cohomology ring of the base $\mathbb{R}P^n$ is generated by a degree-one class $x$ satisfying $x^{n+1}=0$. 
Clearly, $p^*(x)=0$, and by \Cref{prop: cohomological lower bound of m-secat} we obtain $\sct_1(p) \geq n$. 
Note that $\mcat(\mathbb{R}P^n)=\mathrm{cat}(\mathbb{R}P^n)=n$ for all $m$, see \cite[Example 2.5]{m-hom-dist}.
Therefore, by \Cref{prop: monotonicity of m-secat} and \Cref{cor: m-secat leq m-cat}, it follows that $\sctm(p)=\sct(p)=n$ for all $m$.
\end{example}

The following proposition shows that \cite[Proposition 4.1]{MR2879375} holds under weaker conditions.

\begin{proposition}
\label{prop: prod-ineq for m-cat}
If $X$ and $Y$ are path-connected normal spaces, then 
$$
    \mcat(X\times Y)\leq \mcat(X)+\mcat(Y).
$$    
\end{proposition}

\begin{proof}
After natural identification of spaces, we have $e_{X \times Y} = e_X \times e_Y$. 
Hence, the desired result follows from \Cref{cor: m-cat = m-secat(e_X)} and \Cref{prop: prod-ineq for m-secat}.
\end{proof}

\subsection{Bounds on $m$-category}
\hfill\\ \vspace{-0.7em}

We first define the $m$-analogue of the cup length of a topological space.

\begin{definition}
Let $X$ be a topological space, $R$ a commutative ring with unity, and $m\geq 1$ an integer. 
The \emph{$m$-cup length} of $X$ (with coefficients in $R$), denoted by $\cl_R^m(X)$, is the largest integer $k\geq 0$ such that there exist cohomology classes $u_1,\ldots,u_k \in \widetilde{H}^*(X;R)$ such that 
\[
    1\le \deg(u_i)\le m \quad \text{for all } 1\le i\le k,
    \quad \text{and} \quad
    u_1\smile\cdots\smile u_k\neq 0.
\]
If $X$ has no nonzero cohomology class of degree between $1$ and $m$, then $\cl_R^m(X)=0$.
\end{definition}

The following proposition gives several properties of the $m$-cup length.

\begin{proposition}
\label{prop: properties of m-cup length}
Let $X$ be a topological space and $R$ a commutative ring. 
Then the following statements hold.
\begin{enumerate}
\item If $\cl_R(X)$ is the cup-length of $X$ with coefficients in $R$, then
$$
    \cl_R^1(X) \leq \cl^2_{R}(X) \leq \dots \leq \cl_R^m(X) \leq \dots \leq \cl_R(X).
$$

\item $\cl_R^i(X) = 0$ for all $i\leq \conn(X)$.
    
\item If the cohomology ring $H^*(X;R)$ is finitely generated as a graded $R$-algebra and $\{\alpha_l\}_{l\in \lambda}$ is a generating set, then $\cl_R^i(X) = \cl_R(X)$ for all $i\geq \max\{\deg(\alpha_l): l\in \lambda\}$.

\item If $X$ has the homotopy type of a CW complex, then $\cl_R^i(X) = \cl_R(X)$ for all $ i \geq \hdim(X)$.


\item If $Y$ is a topological space such that $H^k(Y;R)$ is a finitely generated free $R$-module for all $k$, then
\[
    \cl_R^m(X\times Y)
        = \cl_R^m(X) + \cl_R^m(Y).
\]
\end{enumerate}
\end{proposition}

\begin{proof}
(2) By the Hurewicz theorem, followed by the universal coefficient theorem for cohomology, we obtain $H^{i}(X;R) = 0$ for all $i \leq \conn(X)$.
Hence, the first assertion follows.

(3) 
It is enough to show that $\cl_R(X)\leq \cl_R^m(X)$, where $m = \max \{\deg(\alpha_l): l\in \lambda\}$. Suppose that there exist cohomology classes $\beta_1,\ldots, \beta_k\in \widetilde{H}^*(X;R)$ such that $\beta_1\smile \cdots \smile \beta_k \neq 0$. In particular, every class $\beta_i$ is obtained as a finite sum of products of generators, each having degree at most $m$. 
Applying the distributive property of the cup product, we obtain that at least one non-vanishing monomial term consists of at least $k$ many cohomology classes. 
It follows that $k\leq \cl_R^{m}(X)$. 

(4) Since $H^{i}(X;R) = 0$ for all $i \geq \hdim(X)$, the assertion follows.

(5) Choose classes $u_1,\ldots,u_r \in \widetilde{H}^*(X;R)$ and $v_1,\ldots,v_s\in \widetilde{H}^{*}(Y;R)$ such that 
$$
    1 \leq \deg(u_i),\deg(v_j) \leq m, 
        \quad u_1\smile\cdots\smile u_r\neq0,
        \quad \text{and} \quad
        v_1\smile\cdots\smile v_s\neq0.
$$
Let $p_X \colon X\times Y\to X$ and $p_Y \colon X\times Y\to Y$ denote the projection maps.
Consider the classes
\[
    p_X^*(u_1),\ldots,p_X^*(u_r), p_Y^*(v_1) ,\ldots, p_Y^*(v_s) 
    \in \widetilde{H}^*(X \times Y;R) \,.
\]
Under the K\"unneth isomorphism, $H^*(X\times Y;R)\cong H^*(X;R)\otimes H^*(Y;R)$, the product 
\[
    p_X^*(u_1)\smile\cdots\smile p_X^*(u_r)
        \smile p_Y^*(v_1)\smile\cdots\smile p_Y^*(v_s)
        = p_X^*(u_1\smile\cdots\smile u_r)
        \smile p_Y^*(v_1\smile\cdots\smile v_s)
\]
in $H^*(X\times Y;R)$ corresponds to the product
\[
    (u_1\smile\cdots\smile u_r) \otimes (v_1\smile\cdots\smile v_s) 
        \in H^*(X;R)\otimes H^*(Y;R),
\]
which is nonzero. 
Hence, $\cl_R^m(X\times Y) \geq \cl_R^m(X) + \cl_R^m(Y)$.

For the reverse inequality, let $w_1,\ldots,w_t\in\widetilde{H}^*(X\times Y;R)$ such that
\[
    1 \leq \deg(w_k) \leq m, 
        \quad
    w_1\smile\cdots\smile w_t\neq0
\] 
By the K\"unneth theorem, each $w_k$ is a finite sum of tensors
$a\otimes b$, where
\[
    a \in H^p(X;R), \qquad b \in H^q(Y;R), \qquad 1\leq p+q = \deg(w_k) \leq m.
\]
Expanding the cup product, there is a nonzero summand of the form
\[
    (a_1\smile\cdots\smile a_t) \otimes (b_1\smile\cdots\smile b_t).
\]
Let $I=\{i\mid \deg(a_i)>0\}$ and $J=\{j\mid \deg(b_j)>0\}$.
Since $1\leq \deg(a_k)+\deg(b_k)\leq m$ for every $1 \leq k \leq t$, every positive-degree class
$a_i$ ($i \in I$) and $b_j$ ($j \in J$) has degree at most $m$. 
Moreover,
\[
    \smile_{i \in I} a_i \neq 0,
        \qquad
    \smile_{j \in J} b_j \neq 0,
\]
implies $|I|\leq \cl_R^m(X)$ and $|J|\leq \cl_R^m(Y)$.
Since $\deg(a_k)+\deg(b_k) > 0$ for every $1 \leq k \leq t$, each index $k$ belongs to at least one of $I$ or $J$.
Thus, 
\[
    t \leq |I|+|J| \leq \cl_R^m(X) + \cl_R^m(Y).
\]
Hence, $\cl_R^m(X\times Y) \leq \cl_R^m(X) + \cl_R^m(Y)$.
\end{proof}

\begin{proposition}
\label{prop: cohomological lower bound on m-cat}
For a path-connected topological space $X$, we have
$$
     \cl_R^{m}(X) \leq \mcat(X)
$$
for all $m \geq 1$ and any commutative ring $R$.
\end{proposition}

\begin{proof}
Suppose that there exist cohomology classes $u_0, u_1, \dots,u_k \in H^*(X;R)$ such that 
\[
    1\leq \deg(u_i)\leq m \quad \text{for all } 0\leq i\leq k,
    \quad \text{and} \quad
    u_0 \smile u_1\smile\cdots\smile u_k\neq 0.
\]
By \Cref{cor: m-cat = m-secat(e_X)}, it is enough to show that $\sctm(e_X) \geq k+1$. 
Since $P_{x_0}X$ is contractible, it follows that $e_X^{*}(u_0) = e_X^{*}(u_1) = \cdots = e_X^{*}(u_k) = 0$. 
Hence, the result follows from \Cref{prop: cohomological lower bound of m-secat}.
\end{proof}

\begin{proposition}
\label{prop: m-cat = cat if X is ''aspherical''}
Suppose $X$ is a path-connected CW complex such that $\pi_{k+1}(X) = 0$ for all $k \geq m$, then 
$$
    \mcat(X) = \ct(X).
$$
In particular, if $X$ is aspherical, then $\mcat(X) = \ct(X)$ for all $m$.
\end{proposition}

\begin{proof}
By \Cref{cor: m-cat = m-secat(e_X)} and the fact $\ct(X) = \sct(e_X)$, it is enough to show that $\sctm(e_X) = \sct(e_X)$. 
As the fibre $\Omega X$ of $e_X$ satisfies $\pi_{k}(\Omega X) = \pi_{k+1}(X) = 0$ for all $k \geq m$, the desired equality follows from \Cref{prop: m-secat = secat if F is ''apsherical''}.
\end{proof}

The following proposition follows from \Cref{prop: secat leq m-secat + [k/m+1]} and \Cref{cor: m-cat = m-secat(e_X)}.

\begin{proposition}
\label{prop: cat leq mcat  + [k/m+1]}
Suppose $X$ is a path-connected $k$-dimensional CW complex. 
Then
$$
    \ct(X) \leq \mcat(X) + \left[\frac{k}{m+1}\right], 
        \quad \text{and} \quad
    \ct(X) \leq \max\{\ct_{k-1}(X),2\},
$$
where $[w]$ denotes the greatest integer less than or equal to $w$.
\end{proposition}

Using similar arguments to those in \Cref{cor: secat leq m-secat + [hdim/m+1]}, together with previous proposition and the invariance of $m$-category proved later in \Cref{cor: invariance of m-cat}, we get the following result.

\begin{corollary}
\label{cor: cat leq mcat + [hdim/m+1]}
Suppose $X$ is a path-connected finite-dimensional CW complex. 
Then
$$
    \ct(X) \leq \mcat(X) + \left[\frac{\hdim(X)}{m+1}\right], 
$$
where $[w]$ denotes the greatest integer less than or equal to $w$.
In particular, 
$$
    \mcat(X) = \ct(X)
$$ 
for all $m\geq \hdim(X)$.   
\end{corollary}

\begin{example}
\label{ex: m-cat of Moore spaces}
Consider the Moore space $M(G,n)$, where $G$ is a finitely generated abelian group and $n\geq 2$. 
It is known that $\dim(M(G,n)) = n$ if $G$ is torsion-free, and $\dim(M(G,n)) = n+1$ otherwise. 
We note that $\mcat(M(G,n)) = 0$ for $m < n$, since it is a $(n-1)$-connected CW complex. 
From \Cref{cor: cat leq mcat + [hdim/m+1]}, we have $\mcat(M(G,n)) = \ct(M(G,n)) = 1$ for $m \geq n$ when $G$ is torsion-free; see \cite[Example 1.33]{CLOT} for the LS category of Moore spaces. 
Similarly, if $G$ has torsion, then $\mcat(M(G,n)) = \ct(M(G,n)) = 1$ for $m\geq n+1$. 
However, $\mcat(M(G,n)) \neq 0$ for $m=n$, because the map $\iota \circ \phi \colon S^n \to M(G,n)$ is not nullhomotopic. 
Here $\iota \colon S^n \hookrightarrow M(G,n)$ is an inclusion map and $\phi \colon S^n\to S^n$ is the attaching map of a $(n+1)$-cell in $M(G,n)$, whose degree corresponds to the order of a summand of the torsion subgroup of $G$. 
Moreover, $M(G,n) = \Sigma M(G,n-1) = CM_U(G,n-1)\cup CM_L(G,n-1)$, where $CM_U(G,n-1)$ and $CM_L(G,n-1)$ denote the open upper and lower cones in the suspension $\Sigma M(G,n-1)$, respectively. 
This implies $\mcat(M(G,n))= 1 =\ct(M(G,n))$ for $m=n$. 
Hence,
$$
\mcat(M(G,n)) =
\begin{cases}
    0 & \text{ if } m<n,\\
    1 & \text{ if } m\geq n.
\end{cases}
$$
\end{example}

\begin{example} 
\label{ex: mcat of suspension}
Consider the space $\Sigma X\times Y$, where $Y$ is a $n$-connected CW complex. 
Observe that $\Sigma X \times Y = (CX_U\times Y)\cup (CX_L\times Y),$ where $CX_U$ and $CX_L$ denote the open upper and lower cones in the suspension $\Sigma X$, respectively.
Then $\mcat(\Sigma X\times Y)\leq 1$ for all $m \leq n$. 
Indeed, by the cellular approximation theorem, every map $\phi_1\colon P\to CX_U\times Y$ and $\phi_2\colon P\to CX_L\times Y$ from an $m$-dimensional CW complex $P$ is nullhomotopic whenever $m\leq n$. 

Further, suppose that $X$ is $k$-connected CW complex, where $1 \leq k < n-1$.
Since $\Sigma X$ is $(k+1)$ connected, it follows that $\Sigma X\times Y$ is $(k+1)$ connected. 
Thus, $\mcat(\Sigma X\times Y) = 0$ for all $m \leq k+1$.
Further, by \Cref{prop: cohomological lower bound on m-cat}, we obtain
$
    1 \leq \ct_{k+2}(\Sigma X\times Y)
        \leq \mcat(\Sigma X\times Y)
        \leq 1
$
for $k+1 < m\leq n$. 
Hence, $\mcat(\Sigma X\times Y) = 1$ for $k+1 < m \leq n$.

Moreover, if $H^i(Y;R)$ (or $H^i(X)$) is finitely generated free $R$-module for all $i$, then $H^*(\Sigma X \times Y;R)\cong H^*(\Sigma X;R)\otimes_R H^*(Y;R)$ by the Künneth theorem. 
Therefore, using \Cref{prop: monotonicity of m-cat} and \Cref{prop: cohomological lower bound on m-cat}, we obtain 
$$
    2 \leq \ct_{n+1}(\Sigma X\times Y)
        \leq \mcat(\Sigma X\times Y)
        \leq \ct(\Sigma X\times Y), ~\text{ for all } m \geq n+1.
$$  
\end{example}

If $X$ and $Y$ are particular CW complexes of the above type, then we can obtain a complete classification of their $m$-category. 
This is illustrated in the following example.

\begin{example}
\label{ex: m-cat of product of spheres}
Suppose $\overline{n}=(n_1, n_2, \dots, n_k) \in \mathbb{N}^k$, such that $n_1 < n_2< \dots<n_k.$ 
Let us denote by $S^{\overline{n}}$ the product of spheres $S^{n_1} \times S^{n_2} \times \dots \times S^{n_k}.$ 
Note that, by \Cref{ex: m-cat of Moore spaces}, we have $\mcat(S^{n_i}) = 0$ if $m < n_i$ and $\mcat(S^{n_i}) = 1$ otherwise.
Since $S^{\overline{n}}$ is $(n_1-1)$-connected, it follows that $\mcat(S^{\overline{n}}) = 0$ for $m < n_1$.
Suppose $a_j \in H^{n_j}(S^{\overline{n}};\mathbb{Q})$ denote the pullback of the fundamental class of $S^{n_j}$ under the projection map $S^{\overline{n}} \to S^{n_j}$.
Then 
$
    \prod_{j=1}^{i} a_j \in H^{*}(S^{\overline{n}}; \mathbb{Q})
$
is nonzero.
Using \Cref{prop: cohomological lower bound on m-cat}, we see that $\mathrm{cat}_{n_i}(S^{\overline{n}}) \geq i$.
Hence, by \Cref{prop: monotonicity of m-cat} (1), it follows that $\mcat(S^{\overline{n}}) \geq i$ for all $m \geq n_i$.
On the other hand, by \Cref{prop: prod-ineq for m-cat}, it follows that 
$$
    \mcat(S^{\overline{n}}) 
        \leq \mcat(S^{n_1}) + \mcat(S^{n_2}) + \cdots + \mcat(S^{n_i})
        = i
$$ for $n_i \leq m < n_{i+1}$. Hence,
$$
\mcat(S^{\overline{n}})= 
\begin{cases}
    0 & \text{ if } m<n_1,\\
    i & \text{ if } n_i\leq m < n_{i+1},\\
    k & \text{ if } m \geq n_{k}.
    \end{cases}
$$
\end{example}

To conclude this section, we present the following proposition, which will be used to establish $m$-analogues of Bochner-type theorems in \Cref{sec: Bochner type}.

\begin{proposition}
\label{prop: m-cat of cover leq m-cat of base}
Let $f \colon \widetilde{X} \to X$ be a covering map. 
If $\widetilde{X}$ is path-connected, then
\[
    \mcat(\widetilde{X})\leq \mcat(X).
\]
\end{proposition}


\begin{proof}
Suppose $U$ is an $m$-categorical open subset of $X$.
Let $V = p^{-1}(U)$ and $\widetilde{f} \colon V \to U$ be the restriction of $f$.
Let $\psi \colon P \to V$ be a map from an $m$-dimensional CW complex $P$. 
Since $U$ is $m$-categorical, the composition $\iota_U \circ (\widetilde{f} \circ \psi) \colon P \to X$ is nullhomotopic. 
Let $H \colon P\times I \to X$ be a homotopy such that $H_0 = \iota_U \circ \widetilde{f} \circ \psi$ and $H_1 = c_{x_0}$ for some $x_0 \in X$. 
Since $p$ is a covering map and $f \circ \iota_V \circ \psi = \iota_U \circ \widetilde{f} \circ \psi$, there exists a unique homotopy $\widetilde{H} \colon P\times I\to \widetilde{X}$ such that $\widetilde{H}_0 = \iota_V \circ \psi$ and $p\circ\widetilde{H}=H$.
Since $H_1$ is the constant map, the map $\widetilde{H}_1$ takes values in the discrete fibre $p^{-1}(x_0)$. 
Therefore, on each connected component of $P$, the map $\widetilde{H}_1$ is constant. 
Consequently, $\iota_V \circ \psi$ is homotopic to a constant map on each connected component of $P$.
Hence, $\iota_V \circ \psi$ is nullhomotopic, since $\widetilde{X}$ is path-connected.
Therefore, $V$ is $m$-categorical open subset of $\widetilde{X}$.
\end{proof}

\section{$m$-homotopic distance}
\label{sec: mhomotopic distance}

In this section, we study the $m$-homotopic distance and show that it can be expressed as the $m$-dimensional sectional category of a specific fibration. 
We obtain cohomological and dimension-connectivity bounds for the $m$-homotopic distance. 
Furthermore, we introduce the notion of $m$-cohomological distance, generalizing the cohomological distance recently introduced in \cite{Coh-dist}, and show that it provides a lower bound for the $m$-homotopic distance.

\begin{definition}[{\cite[Definition 2.1]{hom-dist-bw-maps}}]
The \emph{homotopic distance} between two maps $f,g \colon X\to Y$, denoted by $\mathrm{D}(f,g)$, is the smallest integer $k$ such that $X$ admits an open cover $\{U_i\}_{i=0}^{k}$ with the property that
$$
    \left.f\right|_{U_i} \simeq \left.g\right|_{U_i}
        \quad \text{for all }i.
$$ 
If no such open cover exists, we set $\mathrm{D}(f,g) := \infty$.   
\end{definition}

\begin{definition}[{\cite[Definition 2.8]{m-hom-dist}}]
The \emph{$m$-homotopic distance} between two maps $f,g \colon X\to Y$, denoted by $\Dm(f,g)$, is the smallest integer $k$ such that $X$ admits an open cover $\{U_i\}_{i=0}^{k}$, where each $U_i$ has the following property: for any map $\phi \colon P\to U_i$ from an $m$-dimensional CW complex $P$, we have 
$$
    \left.f\right|_{U_i} \circ \phi 
        \simeq \left.g\right|_{U_i}\circ \phi.
$$ 
If no such open cover exists, we set $\Dm(f,g) := \infty$.
\end{definition}

The following statements follow directly from the definition.
\begin{enumerate}
\item $\Dm(f,g) = \Dm(g,f)$.
\item If $f \simeq  f'$ and $g\simeq  g'$, then $\Dm(f,g)=\Dm(f',g')$.
\end{enumerate}

\begin{property}
Suppose $f,g \colon X \to Y$ are continuous maps.
An open subset $U$ of $X$ will be said to satisfy property $D_{m,\,f,\,g}$ if, for every $m$-dimensional CW complex $P$ and every map $\phi \colon P \to U$, we have
$$
    \left.f\right|_{U} \circ \phi 
        \simeq \left.g\right|_{U}\circ \phi.
$$ 
An open cover $\mathcal{U}$ of $B$ will be said to satisfy property $D_{m,\,f,\,g}$ if every element of $\mathcal{U}$ satisfies property $D_{m,\,f,\,g}$.
\end{property}

The proof of the following corollary follows a similar argument to that of \Cref{prop: monotonicity of m-secat}.

\begin{proposition}
\label{prop: monotonicity of m-hom-dist}
Suppose $f,g \colon X \to Y$ are continuous maps.
Then an open subset of $X$ satisfies property $D_{m,\,f,\,g}$ if and only if it satisfies the property $D_{n,\,f,\,g}$ for all $n \leq m$.
In particular, 
$$
    \Dn(f,g) \leq \Dm(f,g)
$$
for all $n \leq m$.
\end{proposition}

First, we establish the following equality that relates the $m$-homotopic distance and the $m$-sectional category. 
This is the $m$-category analogue of \cite[Theorem 2.7]{hom-dist-bw-maps}.

\begin{theorem}
\label{thm: m-hom-dist(f:g) = m-secat(Pi_1)}
Let $f,g \colon X \to Y$ be continuous maps with the following pullback diagram 
\begin{equation}
\label{diag: m-hom-dist(f:g) = m-secat(Pi_1)}
\begin{tikzcd}
{\mathcal{P}(f,g)} \arrow[d, "\Pi_1"'] \arrow[r, "\Pi_2"]   & Y^{I} \arrow[d, "\pi_Y"] \\
X \arrow[r, "{(f,g)}"]                                    & Y \times Y              
\end{tikzcd}
\end{equation}
Then $\Dm(f,g) = \sctm(\Pi_1)$. 
\end{theorem}

\begin{proof}
Let $U$ be an open subset of $X$ that satisfies property $A_m$ with respect to the fibration $\Pi_1$.
Suppose $\phi \colon P \to U$ is a map, where $P$ is an $m$-dimensional CW complex.
Then there exists a map $s \colon P \to \mathcal{P}(f,g)$ such that $\Pi_1 \circ s = \iota_U \circ \phi$. 
Define a homotopy $H \colon P \times I \to Y$ by $H(x,t) :=(\Pi_2 \circ s(x))(t)$. 
Note we have the following commutative diagram
\[
\begin{tikzcd}
&                                & {\mathcal{P}(f,g)} \arrow[d, "\Pi_1"'] \arrow[r, "\Pi_2"] & Y^{I} \arrow[d, "\pi_Y"] \\
P \arrow[r, "\phi"] \arrow[rru, "s", dotted, shift left] \arrow[rrru, "\widetilde{H}", bend left=49] & U \arrow[r, "\iota_{U}", hook] & X \arrow[r, "{(f,g)}"]                                    & Y \times Y              
\end{tikzcd}
\]
where $\widetilde{H} \colon P \to Y^{I}$ is the map given by $\widetilde{H}(x)(t) = H(x,t)$.
Then the following sequence of equalities 
$$
    \pi_Y \circ \Pi_2 \circ s 
        = (f,g) \circ \Pi_1 \circ s
        = (f,g) \circ \iota_U \circ \phi
        = (f \circ \iota_U \circ \phi,\, g \circ \iota_U \circ \phi)
        = (\left.f\right|_{U} \circ \phi, \,\left.g\right|_{U} \circ \phi)
$$
implies that $H(x,0) = (\mathrm{pr}_1 \circ \pi_Y \circ \Pi_2 \circ s)(x) = (\left.f\right|_{U} \circ \phi)(x)$ and $H(x,1) = (\mathrm{pr}_2 \circ \pi_Y \circ \Pi_2 \circ s)(x) = (\left.g\right|_{U} \circ \phi)(x)$, where $\mathrm{pr}_i \colon Y \times Y \to Y$ denotes the projection map onto the $i$th factor.
Therefore, $H$ defines a homotopy between $\left.f\right|_{U} \circ \phi$ and $\left.g\right|_{U} \circ \phi$, i.e., $U$ satisfies property $D_{m,\,f,\,g}$.
Hence, $\Dm(f,g) \leq \sctm(\Pi_1)$.

To establish the reverse inequality, consider an open subset $U$ of $X$ that satisfies property $D_{m,\,f,\,g}$.
Suppose $\phi \colon P \to U$ is a map, where $P$ is an $m$-dimensional CW complex.
Then we have a homotopy $H \colon P \times I \to Y$ such that $H_0 = \left.f\right|_U \circ \phi$ and $H_1 = \left.g\right|_U \circ \phi$.  
Let $\widetilde{H} \colon P \to Y^{I}$ be the map given by $\widetilde{H}(x)(t) = H(x,t)$.
Then $\pi_Y \circ \widetilde{H} = (\left.f\right|_U \circ \phi,\,\left.g\right|_U \circ \phi)$. 
Therefore, by the universal property of pullbacks, there exists a map $s \colon P\to \mathcal{P}(f,g)$ satisfying $\Pi_1 \circ s = \iota_U \circ \phi$ and $\Pi_2 \circ s = \widetilde{H}$. 
Therefore, $U$ satisfies property $A_{m}$ with respect to the fibration $\Pi_1$.
Hence, $\sctm(\Pi_1) \leq \Dm(f,g)$.
\end{proof}

We now present several applications of the preceding theorem.
First, observe that the monotonicity of the $m$-homotopic distance follows directly from \Cref{prop: monotonicity of m-secat} and \Cref{thm: m-hom-dist(f:g) = m-secat(Pi_1)}.
The following result is the $m$-category analogue of \cite[Corollary 3.8]{hom-dist-bw-maps}.

\begin{corollary}
\label{cor: m-hom-dist(f:g) leq m-cat(X)}
Let $f,g \colon X \to Y$ be continuous maps, where $Y$ is path-connected.
Then
$$
    \Dm(f,g) \leq \mcat(X).
$$
In particular, if $X$ is $r$-connected, then $\Dm(f,g) = 0$ for all $m \leq r$.
\end{corollary}

\begin{proof}
As $Y$ is path-connected, it follows that $\pi_Y$ is a surjective fibration.
Hence, the fibration $\Pi_1$ in \Cref{thm: m-hom-dist(f:g) = m-secat(Pi_1)} is surjective.
Thus, the desired inequality 
$
    \Dm(f,g) 
        = \sctm(\Pi_1) 
        \leq \mcat(X)
$
follows from \Cref{thm: m-hom-dist(f:g) = m-secat(Pi_1)} and \Cref{cor: m-secat leq m-cat} (1).
\end{proof}

We note that the assumption of path-connectedness on $Y$ can be omitted in the above corollary, see \cite[Lemma 2.10]{m-hom-dist}. 
The following corollary recovers \cite[Proposition 2.12]{m-hom-dist}.

\begin{corollary}
\label{cor: m-hom-dist(*:id) = m-cat(X)}
For a path-connected topological space $X$, we have 
$$
    \Dm(c_{x_0},\id_X) = \mcat(X),
$$
where $c_{x_0} \colon X \to X$ denotes the constant map that takes the value $x_0 \in X$.
\end{corollary}

\begin{proof}
Note that the following diagram is a pullback
\[
\begin{tikzcd}
P_{x_0}X \arrow[r, hook] \arrow[d, "e_X"']  & X^I \arrow[d, "\pi_X"] \\
X \arrow[r, "{(c_{x_0},\mathrm{Id}_X )}"]             & {X \times X\,.}       
\end{tikzcd}
\]
Hence, the desired equality 
$
    \Dm(c_{x_0}, \id_X) = \sctm(e_X) = \mcat(X)
$
follows from \Cref{thm: m-hom-dist(f:g) = m-secat(Pi_1)} and \Cref{cor: m-cat = m-secat(e_X)}.
\end{proof}

Now we recover the product inequality \cite[Theorem 3.1]{m-hom-dist}.

\begin{corollary}
Suppose $f,g \colon X \to Y$ and $f', g' \colon X' \to Y'$ are maps, where $X$ and $X'$ are normal topological spaces. 
Then
$$
    \Dm(f \times f',g \times g') \leq \Dm(f,g) + \Dm(f',g').
$$
\end{corollary}

\begin{proof}
Consider the product of pullback diagrams \eqref{diag: m-hom-dist(f:g) = m-secat(Pi_1)} corresponding to the maps $(f,g)$ and $(f',g')$
\[
\begin{tikzcd}
{\mathcal{P}(f,g) \times \mathcal{P}(f',g')} \arrow[d, "\Pi_1 \times \Pi_1'"'] \arrow[rr, "\Pi_2 \times \Pi_2'"] \arrow[drr, phantom, "\square"] &  & Y^{I} \times (Y')^{I} \arrow[d, "\pi_Y \times \pi_{Y'}"] \arrow[r, "\widetilde{h}"] & (Y \times Y')^I \arrow[d, "\pi_{Y \times Y'}"] \\
X \times X' \arrow[rr, "{(f,g) \times (f',g')}"]                                                                 &  & (Y \times Y) \times (Y' \times Y') \arrow[r, "h"]                                   & {(Y \times Y') \times (Y \times Y')\,,}       
\end{tikzcd}
\]
where $\widetilde{h}$ and $h$ are natural homeomorphisms. 
Since
$$
    h \circ ((f,g) \times(f',g')) = (f \times f' , g \times g'),
$$
the desired inequality
$$
    \Dm(f \times f' , g \times g') 
        = \sctm(\Pi_1 \times \Pi_1')
        \leq \sctm(\Pi_1) + \sctm(\Pi_1')
        = \Dm(f,g) + \Dm(f',g')
$$
follows from \Cref{thm: m-hom-dist(f:g) = m-secat(Pi_1)} and \Cref{prop: prod-ineq for m-secat}.
\end{proof}

\subsection{Properties}
\label{subsec: properties of m homotopic distance}

\begin{proposition}[{\cite[Proposition 2.9]{m-hom-dist}}]
\label{prop: m-hom-dist under composition}
Let $f,g \colon X \to Y$ be continuous maps. 
Then for continuous maps $h \colon Y \to Z$ and $k \colon Z \to X$, we have
\begin{enumerate}
    \item $\Dm(h \circ f, \, h \circ g) \;\leq\; \Dm(f,g)$.
    \item $\Dm(f \circ k, \, g \circ k) \;\leq\; \Dm(f,g)$.
    \item $\Dm(f,g) \leq \mathrm{D}(f,g)$.
\end{enumerate}
\end{proposition}

\begin{proposition}[{\cite[Theorem 2.11]{m-hom-dist}}]
Suppose $x_0 \in X$ is a fixed point.
If $\iota_1, \iota_2 \colon X \to X \times X$ are the inclusions defined as $\iota_1(x) = (x,x_0)$ and $\iota_2(x) = (x_0,x)$, then
$$
    \Dm(\iota_1,\iota_2) = \mcat(X).
$$        
\end{proposition}

The following result is the $m$-category analogue of \cite[Lemma 3.18]{hom-dist-bw-maps}.

\begin{proposition}
\label{prop: m-hom-dist under product}
Let $f,g \colon X \to Y$ and $h \colon X' \to Y'$ be continuous maps. 
Then 
$$
    \Dm(h \times f, h \times g) = \Dm(f \times h, g \times h) = \Dm(f,g).
$$
 
\end{proposition}

\begin{proof}
If $U$ is an open subset of $X$ that satisfies property $D_{m,\,f,\,g}$,
then it is a straightforward check that $U \times X'$ is an open subset of $X \times X'$ that satisfies property $D_{m,\,f \times h,\,g \times h}$. 
Hence, $\Dm(f \times h, g \times h) \leq \Dm(f,g)$.

Suppose $x_0' \in X'$.
Let $\iota_1 \colon X \to X \times X'$ denote the inclusion map given by $\iota_1(x) = (x,x_0')$, and let $\pr_1 \colon X \times X' \to X$ denote the projection onto the first coordinate.
Then $f = \pr_1 \circ (f \times h) \circ \iota_1$ and $g = \pr_1 \circ (g \times h) \circ \iota_1$ implies 
$$
    \Dm(f,g) 
        = \Dm(\pr_1 \circ (f \times h) \circ \iota_1, g = \pr_1 \circ (g \times h) \circ \iota_1)
        \leq \Dm(f\times h, g \times h)
$$
by \Cref{prop: m-hom-dist under composition}
\end{proof}

The following theorem generalizes \Cref{cor: m-hom-dist(*:id) = m-cat(X)}.

\begin{proposition}
\label{prop: m-hom-dist(f:*) = m-cat(f)}
Let $f \colon X \to Y$ be a continuous map, where $Y$ is path-connected. 
Then
$$
    \Dm(c_{y_0},f) = \mcat(f),
$$
where $c_{y_0} \colon X \to Y$ denotes the constant map that takes the value $y_0 \in Y$.
Consequently,
$$
    \mcat(f) \leq \min \{ \mcat(X), \mcat(Y)\}.
$$
\end{proposition}

\begin{proof}
Let $U$ be an open subset of $X$. 
Suppose $\phi \colon P \to U$ is a map, where $P$ is an $m$-dimensional CW complex.
Then $f \circ \iota_U \circ \phi$ is nullhomotopic if and only if $f \circ \iota_U \circ \phi \simeq c_{x_0} \circ \iota_U \circ \phi$, since $Y$ is path-connected.
Hence, $\Dm(c_{y_0},f) = \mcat(f)$.

Now we show the inequality.
The first inequality $\mcat(f) = \Dm(f,c_{y_0}) \leq \mcat(X)$ follows from \Cref{cor: m-hom-dist(f:g) leq m-cat(X)},
and the second inequality
$$
    \mcat(f) 
        = \Dm(c_{y_0},f)
        = \Dm(c_{y_0} \circ f, \id_Y \circ f) 
        \leq \Dm(c_{y_0}, \id_Y)
        = \mcat(Y)
$$
follows from \Cref{prop: m-hom-dist under composition} and \Cref{cor: m-hom-dist(*:id) = m-cat(X)}.
\end{proof}

The following result is the $m$-category analogue of \cite[Corollary 3.5]{hom-dist-bw-maps}.

\begin{proposition}
Let $f,g \colon X \to Y$ be continuous maps, where $Y$ is a path-connected space.
Then 
$$
    \Dm(f,g) \leq (\mcat(f)+1)\cdot (\mcat(g)+1) - 1.
$$
\end{proposition}

\begin{proof}
Let $U$ be an open subset of $X$ that satisfies property $C_{m,\,f}$.
Suppose $\phi \colon P \to U \cap V$ is a map, where $P$ is an $m$-dimensional CW complex. 
Then, for any subset $V$ of $X$, the composition $f \circ \iota_{U \cap V} \circ \phi$ is nullhomotopic, since
$
    f \circ \iota_{U \cap V} \circ \phi
        = f \circ \iota_{U} \circ (j \circ \phi),
$ 
where $j \colon U \cap V \hookrightarrow U$ is the inclusion map.
Similarly, if $V$ is an open subset of $X$ that satisfies property $C_{m,\,g}$, we have $g \circ \iota_{U \cap V} \circ \phi$ is nullhomotopic for any subset $U$ of $X$.
Hence, if $\{U_r\}_{r=0}^{k}$ and $\{V_s\}_{s=0}^{l}$ are open covers of $X$ that satisfy properties $C_{m,\,f}$ and $C_{m,\,g}$, respectively, then $\{U_r \cap V_s\}_{r,s=0}^{k,l}$ is an open cover of $X$ such that 
$$
    f \circ \iota_{U_r \cap V_s} \circ \phi
        \simeq g \circ \iota_{U_r \cap V_s} \circ \phi,
$$
since $Y$ is path-connected.
\end{proof}

The following result is the $m$-category analogue of \cite[Proposition 3.7]{hom-dist-bw-maps}.

\begin{proposition}
Let $h,h' \colon Z\to X$ and $f,g \colon X \to Y$ be maps such that $f \circ h' \simeq g \circ h'$.
Then 
$$
    \Dm(f \circ h, g \circ h) \leq \Dm(h,h').
$$
\end{proposition}

\begin{proof}
Let $U$ be an open subset of $X$ that satisfies property $D_{m,\,h,\,h'}$.
Suppose $\phi \colon P \to U$ is a map, where $P$ is an $m$-dimensional CW complex.
Then $h \circ \iota_U \circ \phi \simeq h' \circ \iota_U \circ \phi$.
Hence,
$$
    f \circ h \circ \iota_U \circ \phi 
        \simeq f \circ h' \circ \iota_U \circ \phi
        \simeq g \circ h' \circ \iota_U \circ \phi 
        \simeq g \circ h \circ \iota_U \circ \phi
$$
implies $U$ satisfies property $D_{m,\,f \circ h,\,g \circ h}$.
\end{proof}

\subsection{Invariance}\label{subsec: invariance}
\hfill\\ \vspace{-0.7em}

The following result is the $m$-category analogue of \cite[Propositions 3.11 and 3.12]{hom-dist-bw-maps} and follows by the same argument. Hence, the proof is omitted and left to the reader.

\begin{lemma}
Let $f,g \colon X \to Y$ be maps.
\begin{enumerate}
    \item If $h \colon Y \to Y'$ is a map with a left homotopy inverse, then
    $$
         \Dm(h \circ f, h \circ g) = \Dm(f, g).
    $$

    \item If $k \colon X' \to X$ is a map with a right homotopy inverse, then
    $$
        \Dm(f \circ k, g \circ k) = \Dm(f, g).
    $$
\end{enumerate}
\end{lemma}

The following result is the $m$-category analogue of \cite[Proposition 3.13]{hom-dist-bw-maps} and follows immediately from the preceding lemma.

\begin{proposition}
\label{prop: invariance of m-hom-dist}
Suppose $f, g \colon X \to Y$ and $f', g' \colon X' \to Y'$ are maps.
If $h \colon X' \to X$ and $k \colon Y \to Y'$ are homotopy equivalences such that $k \circ f \circ h \simeq  f'$ and $k \circ g \circ h \simeq g'$:
\[
\begin{tikzcd}
X \arrow[r, "f", shift left] \arrow[r, "g"', shift right]                   & Y \arrow[d, "k"] \\
X' \arrow[r, "f'", shift left] \arrow[r, "g'"', shift right] \arrow[u, "h"] & {Y'\,.}         
\end{tikzcd}
\]
Then 
$
    \Dm(f, g) = \Dm(f', g').
$
\end{proposition}

\begin{corollary}
\label{cor: invariance of m-cat}
Suppose $X$ is a path-connected topological space.
If $Y$ is a topological space that is homotopy equivalent to $X$, then $\mcat(X) = \mcat(Y)$.
\end{corollary}

\begin{proof}
Suppose $h \colon Y \to X$ is a homotopy equivalence with a homotopy inverse $k \colon X \to Y$.  
Let $x_0 \in X$ be a point. 
Then $k \circ c_{x_0} \circ h = c_{k(x_0)}$ and $k \circ \id_X \circ h \simeq \id_Y$.
Hence, it follows that
$
    \mcat(X) 
        = \Dm(c_{x_0},\id_X)
        = \Dm(c_{k(x_0)},\id_Y)
        = \mcat(Y)
$
by \Cref{cor: m-hom-dist(*:id) = m-cat(X)} and \Cref{prop: invariance of m-hom-dist}.
\end{proof}

\subsection{Bounds on $m$-homotopic distance}
\hfill\\ \vspace{-0.7em}

We first establish the cohomological lower bounds on the $m$-homotopic distance.

\begin{proposition}\label{prop: cohomological lower bound on Dm}
Let $f,g \colon X\to Y$ be maps.
Suppose there exist cohomology classes $u_0,u_1, \dots, u_k \in H^*(Y \times Y; R)$ (for any commutative ring $R$) with 
$$
    \Delta_Y^*(u_0) = \Delta_Y^*(u_1) = \cdots = \Delta_Y^*(u_k) = 0 
        \quad \text{ and } \quad
    (f,g)^*(u_0 \smile u_1 \smile \dots \smile u_k) \neq 0,
$$
where $\Delta_Y \colon Y \to Y \times Y$ denotes the diagonal map.
Then 
$
    \Dm(f,g)\geq k+1,
$
where $m=\max\{\deg(u_i) \mid i=0,1,\dots, k\}$.
\end{proposition}

\begin{proof}
Let $v_i = (f,g)^*(u_i)$. 
Then by considering the induced commutative diagram on the cohomology corresponding to diagram \eqref{diag: m-hom-dist(f:g) = m-secat(Pi_1)}, we have
\[
    \Pi_1^*(v_i)
        = \Pi_1^*((f,g)^*(u_i))
        = \Pi_2^* \circ (\Pi_Y^*(u_i))
        = \Pi_2^* \circ (\Delta_Y^*(u_i))
        = \Pi_2^*(0)=0.
\]
Therefore, $v_i \in \ker(\Pi_1^*)$. 
Moreover, we have 
$$
    v_0 \smile v_1 \smile \dots \smile v_k 
        = (f,g)^*(u_0 \smile u_1 \smile \dots \smile u_k) 
        \neq 0.
$$ 
Then $\Dm(f,g)\geq k+1$ follows from \Cref{prop: cohomological lower bound of m-secat} and \Cref{thm: m-hom-dist(f:g) = m-secat(Pi_1)}.
\end{proof}

\begin{remark}
Observe that the cohomological lower bound on $\TCm(X)$ established in \Cref{prop: cohomological lower bound of m-TC} can be recovered by setting $f=\pr_1$ and $g=\pr_2$ in \Cref{prop: cohomological lower bound on Dm}.    
\end{remark}

\subsubsection{$m$-Cohomological distance}\label{subsubsec: m cohomological distance}
\hfill\\ \vspace{-0.7em}

Authors in \cite{Coh-dist} introduced an algebraic invariant called the \emph{cohomological distance} between continuous maps. 
This invariant provides a lower bound for the homotopic distance. 
We recall the definition and introduce the notion of $m$-cohomological distance and show that it provides a lower bound for the $m$-homotopic distance.

\begin{definition}
The \emph{cohomological distance} between pair of maps $f,g \colon X \to Y$ with coefficients in $R$, denoted by $\HD(f,g;R)$, is the least integer $k$ for which $X$ admits an open cover $\{U_i\}^k_{i=0}$ such that
$$
    (f|_{U_i})^* = (g|_{U_i})^* 
        \colon H^*(Y;R) \to H^*(U_i;R).
$$
for all $i = 0,1,\dots,k$.
\end{definition}

\begin{definition}
The \emph{$m$-cohomological distance} between pair of maps $f,g \colon X \to Y$ with coefficients in $R$, denoted by $\HDm(f,g;R)$, is the least integer $k$ such that $X$ admits an open cover $\{U_i\}^k_{i=0}$, where each $U_i$ has the following property: for any  $m$-dimensional complex $P$ and any map $\phi \colon P \to U_i$, we have 
$$
    \phi^* \circ (f|_{U_i})^* = \phi^* \circ (g|_{U_i})^* 
        \colon H^*(Y;R) \to H^*(P;R).
$$
\end{definition}

Clearly, for any pair of maps $f,g \colon X \to Y$, we have 
\begin{itemize}
    \item $\HDm(f,g;R)\leq \Dm(f,g)$ for all $m$, and
    \item $\HDm(f,g;R) \leq \HD(f,g;R)$ for all $m$.
\end{itemize}

\begin{property}
Let $f,g \colon X \to Y$ be maps.
An open subset $U$ of $X$ will be said to satisfy property $H^*D_{m,\,f,\,g}$ if, for every $m$-dimensional CW complex $P$ and every map $\phi \colon P \to U$, we have
$$
    \phi^* \circ (f|_{U_i})^* = \phi^* \circ (g|_{U_i})^* 
        \colon H^*(Y;R) \to H^*(P;R).
$$
An open cover $\mathcal{U}$ of $X$ will be said to satisfy property $H^*D_{m,\,f,\,g}$ if every element of $\mathcal{U}$ satisfies property $H^*D_{m,\,f,\,g}$.
\end{property}

\begin{proposition}
\label{prop: monotonicity of m-cohom-dist}
Let $f,g \colon X \to Y$ be maps.
An open subset of $X$ satisfies property $H^*D_{m,\,f,\,g}$ if and only if it satisfies property $H^*D_{n,\,f,\,g}$ for all $n \leq m$.
In particular, $\HDn(f,g;R) \leq \HDm(f,g;R)$ for all $n \leq m$.
\end{proposition}

\begin{proof}
Suppose $U$ is an open subset of $B$ that satisfies property $\HDm(f,g;R)$.
Let $\psi \colon Q \to U$ be a map, where $Q$ is a $n$-dimensional CW complex and $n \leq m$. 
If $q_0 \in Q$ is a point, then
$$
    P := Q \vee S^m = \frac{Q \coprod D^m}{q_0 \sim z},  \quad \text{for all }z \in S^{m-1}
$$
is an $m$-dimensional CW complex, and $\psi$ extends to a map $\phi \colon P \to U$, given by $\phi([q]) = \psi(q)$ and $\phi([x]) = \psi(q_0)$ for $q \in Q$ and $x \in D^m$.
Then $\phi^* \circ (f|_{U})^* = \phi^* \circ (g|_{U})^*$.
If $j \colon Q \hookrightarrow P$ is the inclusion map, then $j^* \circ \phi^* \circ (f|_{U})^* = j^* \circ \phi^* \circ (g|_{U})^*$ implies $\psi^* \circ (f|_{U})^* = \psi^* \circ (g|_{U})^*$.
\end{proof}

Following \cite{hom-dist-bw-maps}, we define $\J(f,g;R)$ to be the image of the linear homomorphism $f^*-g^* \colon H^*(Y;R)\to H^*(X;R)$.

\begin{theorem}\label{thm : lower bound for H*Dm}
Let $f,g \colon X \to Y$ be maps.
If there exist cohomology classes $u_0, u_1, \dots, u_k \in \J(f,g;R)$ (for any commutative ring $R$) with $u_0 \smile u_1 \smile \cdots \smile u_k \neq 0$, then
$\HDm(f,g;R) \geq k+1$, where $m = \max\{\deg(u_i) \mid i=0,1,\dots,k\}$.
\end{theorem}

\begin{proof}
Let $\HDmB(f,g) \leq k$, and let $\{U_i\}^k_{i=0}$ be an open cover of $X$ which satisfies property $H^*D_{m,\,f,\,g}$. 
For each $i$, there exists a cohomology class $v_i \in H^*(Y;R)$ such that $u_i = (f^*-g^*)(v_i)$.


We claim that $\iota_0^*(u_0) = \iota_1^*(u_1) = \cdots = \iota_k^*(u_k) = 0$, where $\iota_i \colon U_i \hookrightarrow X$ denotes the inclusion map. 
Suppose $\iota_i^{*}(u_i) \neq 0$ for some $i$.
By the cellular approximation theorem, there exists a CW complex $W_i$ and a weak homotopy equivalence $h_i \colon W_i \to U_i$.
Then $h_i^*(\iota_i^*(u_i)) \neq 0$.
Note that 
$$
    (h_i^* \circ \iota_i^*)(u_i) 
        = (h_i^* \circ \iota_i^* \circ (f^* - g^*))(v_i)
        = (h_i^* \circ (\left.f\right|_{U_i})^* - h_i^* \circ (\left.g\right|_{U_i})^*) (v_i).
$$

\noindent
Case (1): Suppose $\dim(W_i) \leq m$. 
Then, by \Cref{prop: monotonicity of m-cohom-dist}, it follows that $h_i^* \circ (\left.f\right|_{U_i})^* = h_i^* \circ (\left.g\right|_{U_i})^*$. 
Hence, $h_i^*(\iota_i^{*}(u_i)) = 0$, which is a contradiction.

\noindent
Case (2): Suppose $\dim(W_i) > m$.
Let $(W_{i})_m$ denote the $m$-skeleton of $W_i$ and $\iota \colon (W_i)_m \hookrightarrow W_i$ denotes the inclusion map.
Then $\iota^* \colon H^{j}(W_i;R) \to H^{j}((W_i)_m;R)$ is injective for $j \leq m$.
Hence, $m \geq \deg(u_i)$ implies $\iota^*(h_i^{*}(\iota_i^*(u_i))) \neq 0$. 
Using property $H^{*}D_{m,\,f,\,g}$, it follows that 
$$
    (h_i \circ \iota)^* \circ (\left.f\right|_{U_i})^* = (h_i \circ \iota) \circ (\left.f\right|_{U_i})^*.
$$
Hence, 
$\iota^*(h_i^*(\iota_i^{*}(u_i))) 
    = \iota^* ((h_i^* \circ (\left.f\right|_{U_i})^* - h_i^* \circ (\left.g\right|_{U_i})^*) (v_i))
    = 0
$
gives a contradiction.

Using the long exact sequence in cohomology for the pair $(X,U_i)$:
\[
\begin{tikzcd}
\cdots \arrow[r] & {H^{*}(X,U_i;R)} \arrow[r, "j_i^*"] & H^{*}(X;R) \arrow[r, "\iota_i^*"]                                                                     & H^{*}(U_i;R) \arrow[r] & \cdots \\
                 &                                     & {H^{*}(Y;R)\,,} \arrow[ru, "(\left.f\right|_{U_i})^*-(\left.g\right|_{U_i})^*"'] \arrow[u, "f^*-g^*"] &                        &       
\end{tikzcd}
\]
we see that there exists a relative cohomology class $w_i \in H^*(X, U_i; R)$ such that $j_i^*(w_i) = u_i$, where $j_i \colon X \hookrightarrow (X,U_i)$ is the inclusion map.
Thus, we obtain
$$
    w_0 \smile w_1 \smile \cdots \smile w_k 
        \in H^*(X, \cup^k_{i=0} U_i; R) = H^*(X,X; R)= 0.
$$
Moreover, by the naturality of the cup product, we have 
$$
    u_0 \smile u_1 \smile \cdots \smile u_k 
        = j^*(w_0 \smile w_1 \smile \cdots \smile w_k) = 0,
$$ 
where $j \colon X \hookrightarrow (X,X)$ is the inclusion map.
This contradicts our assumption.
\end{proof}

\subsubsection{Upper bounds}
\hfill\\ \vspace{-0.7em}

In \cite[Theorem 2.19]{m-hom-dist}, the authors established that $\mathrm{D}_1(f,g) = \mathrm{D}(f,g)$ when $Y$ is a $K(\pi,1)$ space. 
The result below generalizes this result.

\begin{proposition}
\label{prop: m-hom-dist = hom-dist if Y is ''aspherical''}
Let $f,g \colon X \to Y$ be maps, where $X$ is a CW complex and $\pi_{k+1}(Y) = 0$ for all $k \geq m$, then
$$
    \Dm(f,g) = \mathrm{D}(f,g). 
$$
In particular, if $Y$ is aspherical, then $\Dm(f,g) = \mathrm{D}(f,g)$ for all $m$.
\end{proposition}

\begin{proof}
By \Cref{thm: m-hom-dist(f:g) = m-secat(Pi_1)}, it is enough to show that $\sctm(\Pi_1) = \sct(\Pi_1)$, where the fibration $\Pi_1$ is from diagram \eqref{diag: m-hom-dist(f:g) = m-secat(Pi_1)}.
As the fibre $\Omega Y$ of $\Pi_1$ satisfies $\pi_k(\Omega Y) = \pi_{k+1}(Y) = 0$ for all $k \geq m$, the desired equality follows from \Cref{prop: m-secat = secat if F is ''apsherical''}. 
\end{proof}

The following proposition follows from \Cref{prop: secat leq m-secat + [k/m+1]} and \Cref{thm: m-hom-dist(f:g) = m-secat(Pi_1)}.

\begin{proposition}
\label{prop: hom-dist leq m-hom-dist + [k/m+1]}
Let $f,g \colon X \to Y$ be maps, where $X$ is a $k$-dimensional CW complex. 
Then
$$
    \mathrm{D}(f,g) \leq \Dm(f,g) + \left[\frac{k}{m+1}\right], 
        \quad \text{and} \quad 
    \mathrm{D}(f,g) \leq \max\{\mathrm{D}_{k-1}(f,g),2\},
$$
where $[w]$ denotes the greatest integer less than or equal to $w$.
\end{proposition}

Using similar arguments to those in \Cref{cor: secat leq m-secat + [hdim/m+1]}, together with previous proposition and the invariance of $m$-homotopic distance established in \Cref{prop: invariance of m-hom-dist}, we get the following result.

\begin{corollary}
\label{cor: hom-dist leq m-hom-dist + [hdim/m+1]}
Let $f,g \colon X \to Y$ be maps, where $X$ is a finite-dimensional CW complex. 
Then
$$
    \mathrm{D}(f,g) \leq \Dm(f,g) + \left[\frac{\hdim(X)}{m+1}\right], 
$$
where $[w]$ denotes the greatest integer less than or equal to $w$.
In particular, 
$$
    \Dm(f,g) = \mathrm{D}(f,g).
$$
for $m \geq \hdim(X)$.
\end{corollary}



\subsection{Normal spaces and triangle inequality}
\hfill\\ \vspace{-0.7em}

The proof of the following proposition follows a similar argument to that of \Cref{prop: inheritance of m-secat}. 
Hence, the proof is omitted and left to the reader.

\begin{proposition}
\label{prop: inheritance of m-hom-dist}
Property $D_{m,\,f,\,g}$ is inherited by open subsets and finite disjoint unions.
\end{proposition}



\begin{lemma}[{\cite[Lemma 4.3]{oprea2011mixing}}]
\label{lemma: merging covers property}    
Let $X$ be a normal space with two open covers
$$
    \mathcal{U} = \{U_0,U_1,\dots,U_k\} 
        \quad \text{and} \quad
    \mathcal{V} = \{V_0,V_1,\dots,V_n\}
$$
such that each set of $\mathcal{U}$ satisfies property (A) and each set of $\mathcal{V}$ satisfies property (B).
If properties (A) and (B) are inherited by open subsets and disjoint unions, then $X$ has an open cover 
$$
    \mathcal{W} = \{W_0,W_1,\dots,W_{k+n}\}
$$
by open sets satisfying both property (A) and property (B).
\end{lemma}

The following result is the $m$-category analogue of \cite[Proposition 3.16]{hom-dist-bw-maps}.

\begin{proposition}
\label{prop: triangle-inequality}
Let $f,g,h \colon X \rightarrow Y$ be maps, where $X$ is a normal space. 
Then
$$
    \Dm(f,g) \leq \Dm(f,h) + \Dm(h,g).
$$
\end{proposition}

\begin{proof}
Assume that $\Dm(f,h) = k$ and $\Dm(h,g) = n$. 
Let $\mathcal{U} = \{U_0, \dots, U_{k}\}$ be an open cover of $X$ that satisfies property $D_{m,\,f,\,h}$, and let $\mathcal{V} = \{V_0, \dots, V_{n}\}$ be an open cover of $X$ that satisfies property $D_{m,\,h,\,g}$.
Then, by \Cref{lemma: merging covers property} and \Cref{prop: inheritance of m-hom-dist}, there exists an open cover $\mathcal{W} = \{W_0, \dots, W_{k+n}\}$ such that
$$
    f \circ \iota_{W_i} \circ \phi 
        \simeq h \circ \iota_{W_i} \circ \phi
        \simeq g \circ \iota_{W_i} \circ \phi
$$
for any map $\phi \colon P \to W_i$, where $P$ is an $m$-dimensional CW complex.
Hence, $\mathcal{W}$ satisfies property $D_{m,\,f,\,g}$.
\end{proof}

\begin{corollary}
Let $f,g \colon X \rightarrow Y$ be maps, where $X$ is a normal and $Y$ is path-connected. 
Then 
$$
    \Dm(f,g) \leq \mcat(f)+\mcat(g).
$$
\end{corollary}

\begin{proof}
The desired inequality 
$
    \Dm(f,g) 
        \leq \Dm(f,c_{y_0}) + \Dm(c_{y_0},g)
        = \mcat(f) + \mcat(g)
$
follows from \Cref{prop: triangle-inequality} and \Cref{prop: m-hom-dist(f:*) = m-cat(f)}. 
\end{proof}

The following serves as the $m$-category analogue of \cite[Proposition 3.17]{hom-dist-bw-maps} and generalizes statements (1) and (2) of \Cref{prop: m-hom-dist under composition}.

\begin{proposition}
Let $f,g,h \colon X \rightarrow Y$ and $f',g' \colon Y \rightarrow Z$ be maps, where $X$ is a normal space. 
Then
$$
    \Dm(f' \circ f, g' \circ g) \leq \Dm(f,g) + \Dm(f',g').
$$
\end{proposition}

\begin{proof}
Suppose $\Dm(f,g) = k$ and $\Dm(f',g') = n$.
If $\mathcal{U} = \{U_0, \dots, U_{k}\}$ is an open cover of $X$ that satisfies property $D_{m,\,f,\,g}$, then $\mathcal{U}$ satisfies property $D_{m,\,g' \circ f,\,g' \circ g}$.
If $\mathcal{V} = \{V_0, \dots, V_{n}\}$ is an open cover of $Y$ that satisfies property $D_{m,\,f',\,g'}$,
then $\{f^{-1}(V_0), \dots, f^{-1}(V_n)\}$ is an open cover of $X$ that satisfies property $D_{m,\,f' \circ f,\, g' \circ f}$.
Then, by \Cref{lemma: merging covers property} and \Cref{prop: inheritance of m-hom-dist}, there exists an open cover $\mathcal{W} = \{W_0, \dots, W_{k+n}\}$ such that
$$
    (f' \circ f) \circ \iota_{W_i} \circ \phi 
        \simeq (g' \circ f) \circ \iota_{W_i} \circ \phi
        \simeq (g' \circ g) \circ \iota_{W_i} \circ \phi
$$
for any map $\phi \colon P \to W_i$, where $P$ is an $m$-dimensional CW complex.
Hence, $\mathcal{W}$ satisfies property $D_{m,\,f' \circ f,\,g' \circ g}$.
\end{proof}



\subsection{Fibrational inequality} 
\hfill\\ \vspace{-0.7em}

In this subsection, we estimate the $m$-homotopic distance between fibre-preserving maps of fibrations in terms of the $m$-homotopic distance of induced maps on the fibres and the LS category of the base. We note that the such expression in case of classical homotopic distance was established by Mac\'ias-Virg\'os, and Mosquera-Lois in \cite{hom-dist-bw-maps}.

Let $p\colon E \to B$ and $p'\colon E' \to B'$ be fibrations. 
Assume that $f,g\colon E \to E'$ are fibre-preserving maps covering the based maps $\bar f,\bar g\colon B \to B'$, respectively; that is, $p'\circ f=\bar f\circ p$,  and $p'\circ g=\bar g\circ p$.
In other words, the following diagram commutes
\[
\begin{tikzcd}
& E \arrow[d, "p"'] \arrow[r] \arrow[r] \arrow[r] \arrow[r, "{f,\,g}"]                            & E' \arrow[d, "p'"]        &                  \\
                    & B \arrow[r, "{\bar{f},\, \bar{g}}"]                                                             & B'  .                      &                 
\end{tikzcd}
\]
Suppose $b_0$ and $b'_0$ are corresponding base points in $B$ and $B'$, respectively. 
Let $F=p^{-1}(b_0)$, $F'=(p')^{-1}(b'_0)$. 
By the commutativity of the above diagram, we obtain induced maps on the fibres $f_0,g_0\colon F \to F'$.

\begin{theorem}
\label{thm: fibration inequality}
Let $f,g\colon E\to E'$ be maps as stated above. 
Then for $m\geq 1$, we have
\[
    \Dm(f,g)+1 \leq \bigl(\Dm(f_0,g_0)+1\bigr) \bigl(\ct(B)+1\bigr).
\]
\end{theorem}

\begin{proof}
Let $U$ be an open subset of $B$ which is categorical in $B$.
Let $H\colon U\times I \to B$ be a homotopy such that $H_0 = \iota_U$ and $H_1 = c_{b_0}$, where $b_0 \in B$ is the base point defining the fibre $F = p^{-1}(b_0)$.
Set $U' = p^{-1}(U)$.
Since $p \colon E\to B$ is a fibration, the homotopy $H$ admits a lift $H'\colon U'\times I \to E$, as shown in the diagram below, such that $H'_0 = \iota_{U'}$ and $p\circ H' = H\circ(p\times \id_I)$.
\[
\begin{tikzcd}
 & U' \arrow[d, hook] \arrow[rr, "\iota_{U'}", hook]                                             &                           & E \arrow[d, "p"] \\
                    & U'\times I \arrow[r, "p\times \id_I"'] \arrow[r] \arrow[r] \arrow[r] \arrow[rru, "H'", dotted] & U\times I \arrow[r, "H"'] & B               
\end{tikzcd}
\]
Since $H_1 = c_{b_0}$, it follows that the image of $H_1'$ lies in $F$. 

Now let $V$ be an open subset of $F$ associated with $\Dm(f_0,g_0)$.
Set $V' = (H'_1)^{-1}(V)$, and denote by $\widetilde{H}_1' \colon V' \to V$ the restriction of $H_1'$.
Since the image of $H_1'$ lies in $F$, it follows that $V'$ is an open subset of $U'$, and hence of $E$.
Let $\phi \colon P \to V'$ be a map from an $m$-dimensional CW complex $P$. 
Since $V$ is associated with $\Dm(f_0,g_0)$, we have
\[
    f_0|_V \circ \widetilde{H}_1' \circ \phi 
        \simeq g_0|_V \circ \widetilde{H}_1' \circ \phi\,.
\]
Let $\iota_V \colon V \hookrightarrow F$, $\iota_{V'} \colon V' \hookrightarrow E$, $\iota_{F} \colon F \hookrightarrow E$, and $\iota_{F'} \colon F' \hookrightarrow E'$ be inclusion maps.
Then we obtain the following:
\begin{align*}
    f|_{V'}\circ\phi
        = f \circ \iota_{V'}\circ\phi
        \simeq f \circ (H_1|_{V'}) \circ \phi
        = f\circ\iota_F\circ\iota_V\circ \widetilde{H}_1' \circ\phi
        = \iota_{F'}\circ f_0 \circ \iota_V \circ \widetilde{H}_1' \circ\phi.
\end{align*}
Similarly, we have $g|_{V'}\circ\phi \simeq \iota_{F'}\circ g_0 \circ \iota_V \circ \widetilde{H}_1' \circ\phi$.
Hence, 
$$
    f|_{V'}\circ\phi \simeq g|_{V'}\circ\phi \,.
$$
Thus, $V'$ is an open subset of $E$ associated with $\Dm(f,g)$.
We now use above arguments for corresponding covers to obtain the desired inequality.
Let $\ct(B) = k$ and $\Dm(f_0,g_0) = l$.
Consider the open covers
\[
    B = U_0\cup\cdots\cup U_k
        \quad \text{and} \quad
    F = V_0\cup\cdots\cup V_l,
\]
where each $U_i$ is categorical in $B$ and each $V_j$ is associated
with $\Dm(f_0,g_0)$. 
For each $i$, let $U_i' = p^{-1}(U_i)$ and let
\[
    V_{ij}' = (H'_i)_1^{-1}(V_j)\subset U_i',
\]
where $H'_i$ is a lift of the chosen homotopy $H_i$ which takes $U_i$ to a point.
The collection
\[
    \left\{V_{ij}'\right\}_{0 \,\leq\, i \,\leq\, k,\; 0 \,\leq\, j \,\leq\, l}
\]
is an open cover of $E$, and each $V_{ij}'$ satisfies $f|_{V_{ij}'}\circ \phi\simeq g|_{V_{ij}'}\circ \phi$ for any map from  $\phi \colon P\to V_{ij}'$, where $P$ is an $m$-dimensional CW complex. 
Hence
\[
    \Dm(f,g)+1 \leq (k+1)(l+1),
\]
which completes the proof.
\end{proof}

We now turn to applications of the fibration inequality established in \Cref{thm: fibration inequality}. 
For a fibration $p\colon E\to B$ with fibre $F$, Varadarajan \cite{V} established the classical inequality
\begin{equation}
\label{eq: Varadarajan inequality}
    \ct(E)+1 \leq (\ct(F)+1)(\ct(B)+1),
\end{equation}
relating the Lusternik--Schnirelmann category of the total space, the
fibre, and the base. 
We obtain analogous inequality for the $m$-dimensional category.

\begin{corollary}
\label{cor:catm-fibration}
Let $p\colon E\to B$ be a fibration with fibre $F$, where $B$ is a path-connected base. 
Then, for every $m\geq 1$,
\[
    \ct_m(E)+1 
        \leq \bigl(\ct_m(F)+1\bigr) \bigl(\ct(B)+1\bigr).
\]
\end{corollary}

\begin{proof}
Apply \Cref{thm: fibration inequality} to $f = \id_E$ and $g=c_{e_0}$, a constant map at $e_0$, where $e_0\in E$. 
The corresponding maps on the base are $\bar{f} = \id_B$ and $\bar{g} = c_{b_0}$, where $b_0 = p(e_0)$. 
Their restrictions to the fibre are $f_0=\id_{p^{-1}(b_0)}$ and $g_0=c_{x_0}$. Note that $F\simeq p^{-1}(b_0)$, since $p$ is a fibration and $B$ is a path connected space. 
Hence, by the definitions of $\ct_m(E)$ and $\ct_m(F)$, the desired inequality follows immediately from \Cref{thm: fibration inequality}.
\end{proof}

\begin{remark}
The following example shows that, in general, the term $\ct(B)$ in \Cref{thm: fibration inequality} cannot be replaced by
$\ct_m(B)$. 
We use \Cref{cor:catm-fibration} for this purpose. 
Consider the principal $S^1$-bundle $S^1 \hookrightarrow L(p,1) \xrightarrow{\pi} S^2$, where $L(p,1) = S^3/\mathbb{Z}_p$ and $\mathbb{Z}_p = \langle \zeta \rangle$ acts on $S^3 \subset \mathbb{C}^2$ by $\zeta \cdot (z_1,z_2) = (\zeta z_1,\zeta z_2)$.
The projection $\pi$ is induced by the Hopf fibration $S^1 \hookrightarrow S^3\to S^2$.
For $p=2$, this gives the principal $S^1$-bundle $S^1\hookrightarrow \mathbb{R}P^3\longrightarrow S^2$.
For $m=1$, we have
\[
    \ct_1(\mathbb{R}P^3)=3,
        \qquad \ct_1(S^1)=1,
        \qquad \ct_1(S^2)=0.
\]
If one were to replace $\ct(B)$ by $\ct_m(B)$ in \Cref{cor:catm-fibration}, then the resulting inequality in this example would give
\[
    \ct_1(\mathbb{R}P^3)+1 
        \leq \bigl(\ct_1(S^1)+1\bigr) \bigl(\ct_1(S^2)+1\bigr),
\]
and hence $4\leq 2$, which is a contradiction. 
Therefore, the term $\ct(B)$ in \Cref{thm: fibration inequality} cannot, in general, be replaced by $\ct_m(B)$.
\end{remark}

\section{$m$-Topological complexity of a space}
\label{sec: mTC}

The notion of $m$-topological complexity of a space $X$, denoted by $\TC^m(X)$, was defined by Mac\'ias-Virg\'os, Mosquera-Lois, and Oprea in \cite{m-hom-dist} as the $m$-homotopic distance between projections. 
In this section, we define it as the $m$-sectional category of the free path space fibration and show that both the definitions are equivalent.

\begin{definition}
The \emph{$m$-topological complexity} of a topological space $X$, denoted by $\TC^m(X)$, is defined as
$$
    \TC^{m}(X) := \sctm\left(\pi_X \colon PX \to X \times X\right).
$$
\end{definition}

The following result recovers the original definition $m$-topological complexity given in \cite[Definition 2.13]{m-hom-dist}.

\begin{proposition}
\label{prop: m-TC(X) = m-hom-dist(pr_1:pr_2)}
For a topological space $X$,
$$
    \TC^m(X) = \Dm(\mathrm{pr}_1, \mathrm{pr}_2),
$$
where $\mathrm{pr}_i \colon X \times X \to X$ denotes the projection map onto the $i$th coordinate.
\end{proposition}

\begin{proof}
Note that $(\mathrm{pr}_1, \mathrm{pr}_2) = \id_{X \times X}$. 
Therefore, the fibration $\Pi_1 \colon \mathcal{P}(\mathrm{pr}_1, \mathrm{pr}_2) \to X \times X$ in \Cref{thm: m-hom-dist(f:g) = m-secat(Pi_1)} is equivalent to $\pi_X \colon X^I \to X \times X$. 
Hence, the desired equality
$$
    \Dm(\mathrm{pr}_1, \mathrm{pr}_2) 
        = \sctm(\Pi_1) 
        = \sctm(\pi_X) 
        = \TC^m(X)
$$
follows from \Cref{thm: m-hom-dist(f:g) = m-secat(Pi_1)}.
\end{proof}

\begin{proposition}
\label{prop: invariance of m-TC}
If $X$ is homotopy equivalent to $Y$, then $\TCm(X) = \TCm(Y)$.
\end{proposition}

\begin{proof}
Suppose $h \colon Y \to X$ is a homotopy equivalence with a homotopy inverse $k \colon X \to Y$.
Then $k \circ \pr_{i,X} \circ (h \times h) \simeq \pr_{i,Y}$ for $i = 1,2$.
Hence, it follows that
$$
    \TCm(X) 
        = \Dm(\pr_{1,X},\pr_{2,X}) 
        = \Dm(\pr_{1,Y},\pr_{2,Y})
        = \TCm(Y)
$$
by \Cref{prop: invariance of m-hom-dist}.
\end{proof}

\begin{proposition}
\label{prop: prod-ineq for m-TC}
If $X \times X$ and $Y \times Y$ are normal spaces, then
$$
    \TC^m(X\times Y) \leq \TC^m(X) + \TCm(Y).
$$
\end{proposition}

\begin{proof}
After natural identification of spaces, we have $\pi_{X \times Y} = \pi_{X} \times \pi_{Y}$. 
Hence, the desired result follows from \Cref{prop: prod-ineq for m-secat}.
\end{proof}


The following proposition, establishing the monotonicity of $\TC^m$, follows directly from \Cref{prop: monotonicity of m-secat}.

\begin{proposition}
\label{prop: monotonicity of m-TC}
Suppose $X$ is a topological space. 
Then
\begin{enumerate}
    \item $\TC^n(X) \leq \TC^m(X)$ for all $n \leq m$.
    \item $\TC^m(X) \leq \TC(X)$ for all $m$.
\end{enumerate}
\end{proposition}

\begin{proposition}
\label{prop: m-cat(X) leq m-TC(X) leq m-cat(X^2)}
For a topological space $X$, we have $\TCm(X) \leq \mcat(X \times X)$.
Moreover,
\begin{enumerate}
\item if $X$ is path-connected, then $\mcat(X) \leq \TCm(X)$.

\item if $X$ is path-connected and normal, then $\TCm(X) \leq 2\,\mcat(X)$. 
Moreover, if $X$ is $r$-connected, then $\TCm(X) = 0$ for all $m \leq r$.
\end{enumerate}
\end{proposition}

\begin{proof}
The inequality $\TCm(X) \leq \mcat(X \times X)$ follows directly from \Cref{cor: m-hom-dist(f:g) leq m-cat(X)}.
If $X$ is path-connected, then by \Cref{cor: m-hom-dist(*:id) = m-cat(X)}, \Cref{prop: m-hom-dist under composition} and, inequality (1) follows
$$
    \mcat(X) 
        = \Dm(c_{x_0},\id_X) 
        = \Dm(\pr_1 \circ \iota_2, \pr_2 \circ \iota_2)
        \leq \Dm(\pr_1,\pr_2)
        = \TCm(X).
$$
In addition, if $X$ is also normal, then inequality (2) follows from \Cref{prop: prod-ineq for m-cat}.
\end{proof}

For topological complexity, Farber and Grant \cite[Lemma~7]{F-G} proved the inequality
\begin{equation}
\label{eq: Farber-Grant inequality}
    \TC(E)+1 \leq (\TC(F)+1)(\ct(B\times B)+1),
\end{equation}
relating the topological complexity of the total space and the fibre with the Lusternik--Schnirelmann category of $B\times B$. 
We next obtain the corresponding inequality for $m$-topological complexity. 

\begin{proposition} 
Let $p\colon E\to B$ be a fibration with fibre $F$, where $B$ is a path-connected base. 
Then, for every $m\geq 1$,
\[
    \TC^m(E)+1 
        \leq \bigl(\TC^m(F)+1\bigr) \bigl(\ct(B\times B)+1\bigr).
\]
\end{proposition}

\begin{proof}
The result follows by applying \Cref{thm: fibration inequality} to the projection maps $\pr_1^E, \pr_2^E\colon E\times E\to E$ and $\pr_1^B,\pr_2^B\colon B\times B\to B$.
\end{proof}

The next result establishes a cohomological lower bound on the $m$-topological complexity.
\begin{proposition}
\label{prop: cohomological lower bound of m-TC}
Suppose $X$ is a topological space such that the homology group $H_i(X;\Z)$ is finitely generated for all $i$. 
Let $\mathbb{K}$ be a field, and let 
$$
    \smile \colon H^*(X;\mathbb{K}) \otimes_{\mathbb{K}} H^*(X;\mathbb{K}) \to H^*(X;\mathbb{K})
$$
denote the cup product homomorphism.
If there exist elements $u_0,u_1,\dots,u_k$ in $\ker(\smile)$ such that $u_0\cdot u_1 \cdots u_k \neq 0$, then 
$$\TC^{m}(X) \geq k + 1, \quad \text{where} \quad m = \max \{\deg(u_i) \mid i = 0,1,\dots,k\}.$$
\end{proposition}

Note that if an element $u_i$ of $H^*(X;\mathbb{K}) \otimes_{\mathbb{K}} H^*(X;\mathbb{K})$ lies in the kernel of $\smile$, it follows that all the summands in the expression of $u_i$ have the same degree. 
Hence, the degree of $u_i$ is well defined.

\begin{proof}
Let $i \colon X \rightarrow PX$ be the homotopy equivalence which maps $x \in X$ to the constant path $c_x$ at $x$. 
If $\Delta_X \colon X \to X \times X$ denotes the diagonal map, then $\pi_X \circ i = \Delta_X$.
Consider the diagram
\[
\begin{tikzcd}
H^*(X;\mathbb{K}) \otimes_{\mathbb{K}} H^*(X;\mathbb{K}) \arrow[r, "\times"] & H^*(X \times X;\mathbb{K}) \arrow[d, "\pi_X^*"'] \arrow[r, "\Delta_X"] & H^*(X;\mathbb{K}) \\
& {H^*(PX;\mathbb{K})\,,} \arrow[ru, "i^{*}"']                           &                  
\end{tikzcd}
\]
where $\times$ is the cross product homomorphism. 
By the K$\ddot{\text{u}}$nneth Theorem \cite[Chapter 7, Theorem 61.6]{munkres2018elements}, the map $\times$ is an $\mathbb{K}$-algebra isomorphism. 
Since $\Delta_X^* \circ \times =\, \smile$ and $\ker (\Delta_X^*) = \ker(\pi_X^*)$, the desired result follows from \Cref{prop: cohomological lower bound of m-secat}.
\end{proof}

\begin{example}
\label{ex: m-TC of complex projective spaces}
Consider the complex projective space $\mathbb{C}P^n$.
From \cite[Theorem 10]{Farber-TC} and \Cref{prop: cohomological lower bound of m-TC}, we obtain $\TC^2(\mathbb{C}P^n)\geq 2n$. 
Moreover, by \Cref{prop: monotonicity of m-TC} (1), we have $$2n \leq \TC^2(\mathbb{C}P^n)\leq \TC^m(\mathbb{C}P^n)\leq \TC(\mathbb{C}P^n) = 2n \quad \text{for all} \quad m \geq 2.$$
Further, using \Cref{prop: m-TC(X) = m-hom-dist(pr_1:pr_2)}, we have $\TC^1(\mathbb{C}P^n) = \mathrm{D}_1(\pr_1,\pr_2)$. 
Since $\mathbb{C}P^n$ is simply connected, we have $\pr_1 \circ \phi \simeq \pr_2\circ \phi$ for any map $\phi\colon P\to \mathbb{C}P^n\times \mathbb{C}P^n$ from a $1$-dimensional CW complex $P$, by the cellular approximation theorem. 
Therefore, $\TC^1(\mathbb{C}P^n) = \mathrm{D}_1(\pr_1,\pr_2) = 0$. 
Hence,
$$
\TC^m(\mathbb{C}P^n) = 
\begin{cases}
    0 & \text{ if } m=1,\\
    2n & \text{ if } m\geq 2.
\end{cases}
$$
\end{example}

\begin{example}
\label{ex: m-TC of real projective spaces}
Consider the real projective space $\mathbb{R}P^n$.
Then we have the inequalities $\mcat(\mathbb{R}P^n)\leq \TC^m(\mathbb{R}P^n)\leq \TC(\mathbb{R}P^n)$ for all $m$. 
We know that $\TC(\mathbb{R}P^n) = n$ for $n=1,3,7$ by \cite[Proposition 18]{Farber-Tabachnikov-YuzvinskyTC}, and $\mcat(\mathbb{R}P^n) = n$ for all $m$ by \cite[Example 2.5]{m-hom-dist}. 
Therefore, if $n=1,3,7$, then we have $\TC^m(\mathbb{R}P^n) = n$ for all $m\geq 1$.

Moreover, if $n=2^{r-1}$ where $r\in \mathbb{N}$, then we obtain the following inequality using \Cref{prop: cohomological lower bound of m-TC} and \cite[Corollary 14]{Farber-Tabachnikov-YuzvinskyTC} $$2^r-1\leq \TC^1(\R P^n)\leq \TCm(\R P^n)\leq \TC(\R P^n)\leq 2^r-1.$$ Therefore, $\TC^m(\mathbb{R}P^n) = 2n-1$ for all $m\geq 1$, whenever $n$ is any power of $2$.
\end{example}

\begin{example}
\label{ex: m-TC of product of spheres}
We revisit \Cref{ex: m-cat of product of spheres}. 
Firstly, note that $\TCm(S^{n}) = 0$ for $m < n$, by \Cref{prop: m-cat(X) leq m-TC(X) leq m-cat(X^2)} (2).
By the zero-divisor cup-length computation of spheres in \cite{Farber-TC}, it follows from \Cref{prop: cohomological lower bound of m-TC} that $\TC^{n}(S^{n}) \geq 1$ if $n$ is odd, and $\TC^{n}(S^{n}) \geq 2$ if $n$ is even.
Hence, by \Cref{prop: monotonicity of m-TC}, it follows that
\begin{equation}
\label{eq: TCm of S^n}
\TCm(S^n) = 
\begin{cases}
    0 & \text{ if } m < n,\\ 
    1 & \text{ if } m \geq n \text{ and } n \text{ odd},\\
    2 & \text{ if } m \geq n \text{ and } n \text{ even}.
\end{cases}
\end{equation}

By \Cref{prop: prod-ineq for m-TC} and \Cref{eq: TCm of S^n}, we have 
$$
    \TCm (S^{\overline{n}}) 
        \leq \TCm (S^{n_1}) + \TCm (S^{n_2}) +  \dots + \TCm (S^{n_i})
        = i + l(i)
$$
for $n_i \leq m < n_{i+1}$, where $l(i)$ denotes the number of even dimensional spheres in the set $\{S^{n_1},S^{n_2},\dots,S^{n_{i}}\}$.
Suppose $a_j \in H^{n_j}(S^{\overline{n}};\mathbb{Q})$ denote the pullback of the fundamental class of $S^{n_j}$ under the projection map $S^{\overline{n}} \to S^{n_j}$.
Then 
$$
    \prod_{j=1}^{i} (1 \otimes a_j - a_j \otimes 1)^{\zeta(j)} \in H^{*}(S^{\overline{n}} \times S^{\overline{n}}; \mathbb{Q})
$$
is nonzero, where $\zeta(j) = 1$ if $n_j$ is odd, and $\zeta(j) = 2$ if $n_j$ is even.
Using \Cref{prop: cohomological lower bound of m-TC}, we see that $\TC^{n_i}(S^{\overline{n}}) \geq i + l(i)$.
Hence, by \Cref{prop: monotonicity of m-TC}, it follows that $\TCm(S^{\overline{n}}) \geq i + l(i)$ for all $m \geq n_i$.
Hence,
$$
\TCm(S^{\overline{n}}) = 
\begin{cases}
    0 & \text{ if } m < n_1,\\ 
    i+l(i) & \text{ if } n_i \leq m < n_{i+1},\\
    k+l(k) & \text{ if } m \geq n_k.
\end{cases}
$$
\end{example}

\begin{example}
\label{ex: m-TC of Moore spaces}
Consider the Moore space $M(G,n)$, where $G$ is a finitely generated abelian group with positive rank and $n \geq 2$. 
From \Cref{ex: m-cat of Moore spaces} and \Cref{prop: m-cat(X) leq m-TC(X) leq m-cat(X^2)} (2), we have the following inequality 
$$
    \TC^m(M(G,n)) \leq 2 \mcat(M(G,n)) = 
    \begin{cases}
        0 & \text{if } m < n,\\
        2 & \text{if } m \geq n.
    \end{cases}
$$ 
The reduced cohomology ring of $M(G,n)$ with coefficients in the field $\mathbb{K}$ is given by
$$
    \widetilde{H}^{j}(M(G,n);\mathbb{K}) = 
\begin{cases}
    \mathbb{K}^r & \text{if } j=n,\\
    0 & \text{otherwise},
\end{cases}
$$
where $r$ is the rank of $G$.
By a computation similar to that of the zero-divisor cup-length of spheres in \cite{Farber-TC}, it follows from \Cref{prop: cohomological lower bound of m-TC} that $\TC^n(M(G,n)) \geq 1$ if $n$ is odd, and $\TC^n(M(G,n)) \geq 2$ if $n$ is even.
Hence,
$$
\TC^m(M(G,n)) = 
\begin{cases}
 0 & \text{if } m<n,\\
 1\text{ or }2 & \text{if }  m\geq n \text{ and } n \text{ odd},\\
 2 & \text{if } m\geq n \text{ and } n \text{ even}.
\end{cases} 
$$
\end{example}

\begin{proposition}
\label{prop: m-TC = TC if F is ''aspherical''}
Suppose $X$ is a CW complex such that $\pi_{k+1}(X) = 0$ for all $k \geq m$, then 
$$
    \TC^m(X) = \TC(X).
$$
In particular, if $X$ is aspherical, then $\TC^m(X) = \TC(X)$ for all $m$.
\end{proposition}

\begin{proof}
Note that it is enough to show that $\sctm(\pi_X) = \sct(\pi_X)$. 
As the fibre $\Omega X$ of $\pi_X$ satisfies $\pi_{k}(\Omega X) = \pi_{k+1}(X) = 0$ for all $k \geq m$, the desired equality follows from \Cref{prop: m-secat = secat if F is ''apsherical''}.
\end{proof}

\begin{example}
\label{ex: m-TC of surfaces}
\normalfont{
The topological complexity of the closed orientable surface $\Sigma_g$ of genus $g$ was computed by Farber in \cite{Farber-TC}. 
Since $\Sigma_g$ is aspherical for all $g \geq 1$, by \Cref{prop: m-cat = cat if X is ''aspherical''} and \Cref{prop: m-TC = TC if F is ''aspherical''}, we obtain
$$  
    \mcat(\Sigma_g) = \ct(\Sigma_g) = 2\text{ for all } g \geq 1,
    \quad \text{and} \quad
    \TC^m(\Sigma_g) = \TC(\Sigma_g) = 
    \begin{cases}
        2 & \text{if }  g=1,\\
        4 & \text{if } g \geq 2,
    \end{cases}
$$ 
for all $m$. 

For the closed non-orientable surface $N_h$ of genus $h$, the topological complexity was computed in \cite{Dr} and \cite{TCKleinbottle}. 
Since $N_h$ is aspherical for all $h \geq 2$, it follows that $\mcat(N_h) = \ct(N_h) = 2$, and $\TCm(N_h) = \TC(N_h) = 4$ for all $h\geq 2$ and all $m$.
}
\end{example}

The following proposition follows from \Cref{prop: secat leq m-secat + [k/m+1]}.

\begin{proposition}
\label{prop: TC leq m-TC + [2k/m+1]}
Suppose $X$ is a $k$-dimensional CW complex. 
Then
$$
    \TC(X) \leq \TCm(X) + \left[\frac{2k}{m+1}\right], 
        \quad \text{and} \quad
    \TC(X) \leq \max\{\TC^{2k-1}(X),2\},
$$
where $[w]$ denotes the greatest integer less than or equal to $w$.
\end{proposition}

Using similar arguments to those in \Cref{cor: secat leq m-secat + [hdim/m+1]}, together with previous proposition and the invariance of $m$-topological complexity established in \Cref{prop: invariance of m-TC}, we get the following result.

\begin{corollary}
\label{cor: TC leq m-TC + [2hdim/m+1]}
Suppose $X$ is a finite-dimensional CW complex. 
Then
$$
    \TC(X) \leq \TCm(X) + \left[\frac{2\hdim(X)}{m+1}\right],
$$
where $[w]$ denotes the greatest integer less than or equal to $w$.
In particular,
$$
    \TCm(X) = \TC(X)
$$
for all $m \geq 2 \hdim(X)$.
\end{corollary}

\begin{proposition}
For continuous maps $f,g \colon X \to Y$, we have
$$
    \Dm(f,g) \leq \min\{\mcat(X), \TCm(Y) \}.
$$ 
\end{proposition}

\begin{proof}
It is enough to show that $\Dm(f,g) \leq \TCm(Y)$, since $\Dm(f,g) \leq \mcat(X)$ follows from \Cref{cor: m-hom-dist(f:g) leq m-cat(X)}.
Hence, the desired inequality 
$$
    \Dm(f,g) = \sctm(\Pi_1) \leq \sctm(\pi_Y) = \TCm(Y)
$$  
follows from \Cref{thm: m-hom-dist(f:g) = m-secat(Pi_1)} and \Cref{prop: m-secat under pullback}.
\end{proof}

\section{H-spaces}\label{sec: H-spaces}

In this section, we study the $m$-homotopic distance between maps whose domain and/or codomain are $H$-spaces.
Recall that a topological space $X$ is called an \emph{$H$-space} if it is equipped with
\begin{itemize}
    \item a map $\mu \colon X \times X \to X$, denoted by $\mu(x,y) := x \cdot y$, called the multiplication, and
    
    
    \item an identity element $x_0 \in X$ 
\end{itemize} 
such that 
    $\mu \circ \iota_1 \simeq \id_X$, where $\iota_1 \colon X \to X \times X$ is the inclusion defined by $\iota_1(x) = (x,x_0)$. 
An $H$-space will be denoted by $(X,\mu, x_0)$.

\begin{definition}
Let $(X,\mu,x_0)$ be an $H$-space. 
A \emph{division} on $X$ is a map $\delta \colon X \times X \to X$, denoted by $\delta(x,y) = x^{-1} \cdot y$, such that $\mu \circ (\pr_1,\delta) \simeq \pr_2\,$; where $\pr_1, \pr_2 \colon X \times X \to X$ denote the projection maps onto the first and the second coordinates, respectively. 
An $H$-space equipped with a division map will be denoted by $(X,\mu,\delta,x_0)$.
\end{definition}

The following theorem is the $m$-category analogue of \cite[Theorem 4.1]{hom-dist-bw-maps}.

\begin{theorem}
\label{thm: H-space m-hom-dist(f:g) leq m-cat(X)}
Let $(X,\mu,\delta,x_0)$ be a path-connected $H$-space and let $f,g \colon X \times X \to X$ be two maps. 
Then 
$$
    \Dm(f,g) \leq \mcat(X).
$$
In particular, $\Dm(\mu,\delta) \leq \mcat(X)$.
\end{theorem}

\begin{proof}
Let $U$ be an open subset of $X$ that satisfies property $B_m$, and let $V:= \left(\delta\circ (f,g)\right)^{-1}(U)$.
If $q \colon Q\to V$ be a map from an $m$-dimensional complex $Q$, then
\begin{align*}
g \circ \iota_V \circ q 
     = \pr_2 \circ (f,g)\circ \iota_V\circ q
     \simeq \mu \circ (\pr_1,\delta) \circ (f,g) \circ \iota_V \circ q
     = \mu \circ (f,\delta\circ(f,g)) \circ \iota_V \circ q.
\end{align*}
Since $\delta \circ (f,g) \circ \iota_V \circ q \colon Q \to X$ factors through $\iota_U$ and $Q$ is an $m$-dimensional CW complex, it follows that 
$
    \delta \circ (f,g) \circ \iota_V \circ q
$
is nullhomotopic. 
In particular, since $X$ is path-connected, it follows that $\delta \circ (f,g) \circ \iota_V \circ q$ is homotopic to the constant map $c_{x_0} \circ f \circ \iota_V \circ q$, where $c_{x_0} \colon X \to X$ is the constant map which takes the value $x_0 \in X$.
Therefore, 
\begin{align*}
\mu \circ (f,\delta\circ(f,g)) \circ \iota_V \circ q
    & ~\simeq~ \mu \circ (f, c_{x_0} \circ f)  \circ \iota_V \circ q\\
    & ~\simeq~ \mu \circ (\id_X, c_{x_0}) \circ f  \circ \iota_V \circ q\\
    & ~=~ \mu \circ \iota_1 \circ f \circ \iota_V \circ q \\
    & ~\simeq~ \id_X \circ f \circ \iota_V \circ q\\
    & ~=~ f \circ \iota_V \circ q.
\end{align*}
Hence, $f \circ \iota_V \circ q \simeq g \circ \iota_V \circ q$, and by applying the argument to open covers, we obtain the desired result.
\end{proof}

The following result is the $m$-category analogue of \cite[Corollary 4.2]{hom-dist-bw-maps}.

\begin{corollary}\label{cor: TCm leq catm for H spaces}
Let $(X,\mu,\delta,x_0)$ be a path-connected $H$-space. Then 
$$
    \TCm(X) = \mcat(X).
$$    
\end{corollary}

\begin{proof}
By \Cref{thm: H-space m-hom-dist(f:g) leq m-cat(X)}, we have 
$
    \TCm(X) = \Dm(\pr_1,\pr_2) \leq \mcat(X).
$
The reverse inequality follows from \cite[Theorem 2.15]{m-hom-dist}.
\end{proof}

The following result is the $m$-category analogue of \cite[Proposition 4.3]{hom-dist-bw-maps}.

\begin{proposition}
\label{prop: H-space m-hom-dist(fh:gh) leq m-hom-dist(f:g)}
Let $f,g,h \colon X \to Y$ be maps, where $(Y,\mu,y_0)$ is a $H$-space. 
Then 
$$
    \Dm(f \cdot h, g \cdot h) \leq \Dm(f,g).
$$  
\end{proposition}

\begin{proof}
Let $\Delta \colon X \to X \times X$ be the diagonal map.
By definition, $f \cdot h = \mu \circ (f \times h)\circ \Delta$ and $g \cdot h = \mu\circ (g\times h)\circ \Delta$. 
Therefore, the desired inequality 
$$
    \Dm(f\cdot h, g\cdot h) = 
        \Dm(\mu\circ (f\times h)\circ \Delta, \mu\circ (g\times h)\circ \Delta)
        \leq \Dm(f\times h, g\times h)
        = \Dm(f,g).
$$
from \Cref{prop: m-hom-dist under composition} and \Cref{prop: m-hom-dist under product}.
\end{proof}

If $X$ is a normal space, the preceding proposition generalizes as follows.

\begin{proposition}
Let $f,g,h,h' \colon X\to Y$ be fibrewise maps, where $X$ is a normal space and $(Y,\mu,x_0)$ is an $H$-space. 
Then
$$
    \Dm(f \cdot h, g \cdot h') \leq \Dm(f,g) + \Dm(h,h').
$$  
\end{proposition}

\begin{proof}
The desired inequality
\begin{align*}
\Dm(f\cdot h, g\cdot h') 
    & = \Dm( \mu \circ (f \times h) \circ \Delta, \mu \circ (g \times h')\circ \Delta)\\
    & \leq \Dm(f \times h, g \times h') \\
    & \leq \Dm(f \times h, g \times h) + \Dm(g \times h, g \times h')\\
    & = \Dm(f,g) + \Dm(h,h').
\end{align*}
follows from \Cref{prop: m-hom-dist under composition}, \Cref{prop: triangle-inequality} and \Cref{prop: m-hom-dist under product}
\end{proof}

\begin{corollary}
For a topological group $G$, we have 
$\Dm(\mu, \delta) = \Dm(\id_G, \mathbf{I}),
$
where $\mathbf{I} \colon G \to G$ is the inversion map defined by $\mathbf{I}(g) = g^{-1}$. 
\end{corollary}

\begin{proof}
Recall that $\mu \circ \iota_1 =\id_G$ and $\delta \circ \iota_1 = \mathbf{I}$. 
Hence, by \Cref{prop: m-hom-dist under composition} (2), we have 
$$
    \Dm(\id_G, \mathbf{I}) 
        = \Dm(\mu \circ \iota_1, \delta \circ \iota_1)
        \leq \Dm(\mu, \delta).
$$
Furthermore,  
$
    \mu = \pr_1 \cdot \pr_2 = (\id_X \circ\, \pr_1) \cdot \pr_2
$ 
and 
$
    \delta = (\mathbf{I} \circ \pr_1) \cdot \pr_2
$
imply that
$$
    \Dm(\mu, \delta) 
        = \Dm((\id_G \circ\, \pr_1) \cdot \pr_2, (\mathbf{I} \circ \pr_1) \cdot \pr_2)
        \leq \Dm(\id_G \circ\, \pr_1, \mathbf{I} \circ \pr_1)
        \leq \Dm(\id_G, \mathbf{I})
$$
by \Cref{prop: H-space m-hom-dist(fh:gh) leq m-hom-dist(f:g)} and \Cref{prop: m-hom-dist under composition} (2).
If $G$ is path-connected, then the inequality $\Dm(\mu, \delta) \leq \mcat(G)$ follows from \Cref{thm: H-space m-hom-dist(f:g) leq m-cat(X)}.
\end{proof}

\begin{corollary}
For a topological group $G$, we have
$\Dm(\mu_a,\mu_b) = \Dm(\mu_{a-b},c_{e}),
$
where $\mu_a \colon G \to G$ is the power map defined by $\mu_a(g) = g^{a}$, and $c_e \colon G \to G$ is the constant map taking the value of the identity element $e \in G$.
\end{corollary}

\begin{proof}
Applying \Cref{prop: H-space m-hom-dist(fh:gh) leq m-hom-dist(f:g)}, we get the following inequalities
$
    \Dm(\mu_a,\mu_b) 
        = \Dm(\mu_{a-b} \cdot \mu_{b}, c_e \cdot \mu_b)
        \leq \Dm(\mu_{a-b},c_e)
$
and
$
    \Dm(\mu_{a-b},c_e) 
        = \Dm(\mu_{a} \cdot \mu_{-b},\mu_{b} \cdot \mu_{-b})
        \leq \Dm(\mu_{a},\mu_{b}).
$
\end{proof}

\begin{example}
Consider the Lie group $G=U(2)$ of $2 \times 2$ complex matrices $A$ such that $A^* = A^{-1}$, and which is topologically $S^1 \times S^3$. 
Note that the integral cohomology ring of $G$ is an exterior algebra generated by a degree one generator $x_1$ and a degree three generator $x_3$ such that $x_1 \smile x_3 \neq 0$.
Let $\id_G \colon G \to G$ be the identity map and let $\mathbf{I} \colon G \to G$ be the inversion map, given by $\mathbf{I}(A) = A^* = A^{-1}$.
Then $\mathbf{I}^{*}(x_1) = -x_1$ and $\mathbf{I}^{*}(x_3) = -x_3$. 
Thus, by \Cref{thm : lower bound for H*Dm}, it follows that $\HD_1(\id_G,\mathbf{I};\mathbb{Z}) \geq 1$ and $\HD_3(\id_G,\mathbf{I};\mathbb{Z}) \geq 2$.
Therefore, the inequalities
$1 \leq \HD_1(\id_G,\mathbf{I};\mathbb{Z}) 
        \leq \HD_2(\id_G,\mathbf{I};\mathbb{Z}) 
        \leq \mathrm{D}_2(\id_G,\mathbf{I};\mathbb{Z})
        \leq \ct_2(G) 
        = 1$
implies $\mathrm{D}_1(\id_G,\mathbf{I}) = \mathrm{D}_2(\id_G,\mathbf{I}) = 1$. 
By \Cref{prop: monotonicity of m-cohom-dist}, we have $\HDm(\id_G,\mathbf{I};\mathbb{Z}) \geq 2$ for all $m \geq 3$.
Hence, the inequalities
$$
    \HDm(\id_G,\mathbf{I};\mathbb{Z}) 
        \leq \Dm(\id_G,\mathbf{I}) 
        \leq \mcat(G) 
        = 2
$$
implies $\Dm(\id_G,\mathbf{I}) = 2$ for all $m \geq 3$. 
Note that $\mathrm{D}(\id_G,\mathbf{I}) = 2$, see \cite[Example 5.4]{hom-dist-bw-maps}.
This shows that the aspherical in higher dimensions assumption in \Cref{prop: m-hom-dist = hom-dist if Y is ''aspherical''} is not strictly necessary.
\end{example}

\section{Theorems of Bochner type}
\label{sec: Bochner type}

Bochner's theorem \cite{MR62505} states that if a compact manifold $B$ admits a metric with non-negative Ricci curvature, then its first Betti number $b_1(B)$ satisfies $b_1(B)\leq \dim(B)$.
Moreover, equality holds precisely when $B$ is a flat torus. 
Using the Cheeger--Gromoll splitting theorem, Oprea \cite{MR1866039} gave a purely topological proof of the stronger estimate $b_1(B)\leq \ct(B)$.
However, the estimate is still sharp only for flat tori; that is, $b_1(B)=\ct(B)$ if and only if $B$ is a flat torus.
Subsequently, Oprea and Strom \cite{MR2879375} strengthened this result by proving that
\begin{equation}
\label{eq: betti num leq 1-cat}
    b_1(B) \leq \ct_1(B).
\end{equation}
In this section, we use the Cheeger--Gromoll splitting theorem to establish a Bochner-type theorem for $\sct_1$, thereby recovering the estimate \eqref{eq: betti num leq 1-cat}. 
Moreover, we also prove $m$-category analogue of Oprea's result \cite[Theorem 3.3]{MR1866039} in \Cref{thm: betti num leq m-cat - m-cup}.

\begin{theorem}[Cheeger--Gromoll splitting {\cite{MR303460}}]
If $B$ is a compact manifold with non-negative Ricci curvature, then there is a finite cover $\widetilde{B}$ of $B$ with a diffeomorphism $\widetilde{B} \cong T^k \times N$, where $N$ is simply connected and $T^k = S^1 \times \dots \times S^1$ ($k$-times) is flat $k$-dimensional torus.    
\end{theorem}

We will refer to such a finite cover $\widetilde{B}$ as a Cheeger-Gromoll splitting of $B$.
If $\pi_1(B)$ is finite, then $\widetilde{B}$ is the universal cover of $B$. 
Consequently, the above theorem is uninteresting in this case.
Hence, we restrict our attention to closed manifolds with non-negative Ricci curvature whose fundamental groups are infinite.

\begin{theorem}
\label{thm: secat geq rank}
Suppose $p \colon E \to B$ is a fibration, where $B$ is a compact manifold with non-negative Ricci curvature and infinite fundamental group.
Then
$$
    \sct_1(p) \geq \mathrm{rank}(\ker(p^* \colon H^1(B;\Q) \to H^1(E;\Q))).
$$
\end{theorem}

\begin{proof}
Suppose $\mathrm{rank}(\ker(p^* \colon H^1(B;\Q) \to H^1(E;\Q))) = l$.
Then there exists $u_1,\dots,u_l \in H^1(B;\Q)$ such that $u_1,\dots,u_l$ are linearly independent and $p^{*}(u_i) = 0$ for all $i$.
By \Cref{prop: cohomological lower bound of m-secat}, it is enough to show that $u_1 \smile \dots \smile u_l \neq 0$.

Let $\pi \colon \widetilde{B} \to B$ be a finite cover of $B$ such that $\widetilde{B}$ is diffeomorphic to $T^k \times N$, where $N$ is simply connected  and $T^k$ is flat $k$-dimensional torus.
Since $\pi^* \colon H^{*}(B;\Q) \to H^{*}(\widetilde{B};\Q)$ is injective (\cite[Lemma 5.6]{MR2879375}), it follows that $\pi^*(u_1), \dots, \pi^*(u_l)$ are linearly independent in $H^1(\widetilde{B};\Q)$. 
Note that
$H^*(T^k \times N;\Q)
        \cong H^*(T^k;\Q) \otimes_{\Q} H^*(N;\Q)
        \cong \Lambda_{\Q}[x_1,\dots,x_k] \otimes_{\Q} H^*(N;\Q),
$
where $\Lambda_{\Q}[x_1,\dots,x_k]$ denotes the exterior algebra over $\Q$ with $k$ generators.
Since $N$ is simply connected, it follows that $H^1(\widetilde{B};\Q) \cong H^{1}(T^k \times N;\Q) \cong H^{1}(T^k;\Q) \subseteq H^*(T^k;\Q) \cong \Lambda_{\Q}[x_1,\dots,x_k]$. Hence, $\pi^*(u_1), \dots, \pi^*(u_l)$ are linearly independent vectors of degree 1 in $\Lambda_{\Q}[x_1,\dots,x_k]$.
Hence, $\pi^*(u_1) \smile \dots \smile \pi^*(u_l) \neq 0$.
Since $\pi^*$ is injective, it follows that $u_1 \smile \dots \smile u_l \neq 0$, and the desired result follows.
\end{proof}

\begin{remark}
\label{remark: cat1(B) geq b1(B)}
Suppose $B$ is a compact manifold with non-negative Ricci curvature and infinite fundamental group.
If $p \colon E \to B$ is a fibration, where $E$ is simply-connected, then, by \Cref{cor: m-secat leq m-cat} (2), it follows that
$$
    \ct_1(B) 
        = \sct_1(p) 
        \geq \mathrm{rank}(\ker(p^* \colon H^1(B;\Q) \to H^1(E;\Q)))
        = \mathrm{rank}(H^1(B;\Q))
        = b_1(B),
$$
where $b_1(B)$ is the first Betti number of $B$.
\end{remark}

Using the preceding remark, we show that \Cref{thm: secat geq rank} recovers \cite[Theorem 5.7]{MR2879375}. 

\begin{corollary}
If $X$ is a compact path-connected manifold with non-negative Ricci curvature and infinite fundamental group, then $b_1(X) \leq \ct_1(X)$, where $b_1(X)$ is the first Betti number of $X$.
\end{corollary}

\begin{proof}
Consider the path space fibration $e_X \colon P_{x_0}X \to X$, given by $\alpha \mapsto \alpha(1)$.
Since the path space $P_{x_0}X$ is contractible, the \Cref{remark: cat1(B) geq b1(B)} implies the desired inequality.
\end{proof}

\begin{remark}
Consider a fibration $F \hookrightarrow E \xrightarrow{p} M$ with connected fibre $F$.
Since $F$ is connected, the long exact sequence of homotopy groups of the fibration yields a surjection $p_\# \colon \pi_1(E)\to \pi_1(M)$.
Passing to abelianizations, we obtain a surjective homomorphism $p_* \colon H_1(E;\mathbb{Z}) \to H_1(M;\mathbb{Z})$.
Since $\mathbb{Q}$ is an injective $\mathbb{Z}$-module, the functor $\Hom(-,\mathbb{Q})$ is exact. 
Hence, the morphism
$\Hom(-,\mathbb{Q})(p_*) \colon
        \Hom\left(H_1(M;\mathbb{Z}),\mathbb{Q}\right) 
        \to \Hom\left(H_1(E;\mathbb{Z}),\mathbb{Q}\right)$
is injective. 
By the naturality of the universal coefficient theorem, we conclude that $p^* \colon H^1(M;\mathbb{Q}) \to H^1(E;\mathbb{Q})$ is injective.
Therefore, $\ker\left(p^* \colon H^1(M;\mathbb{Q})\to H^1(E;\mathbb{Q})\right)=0$.  
Hence, the interesting applications of \Cref{thm: secat geq rank} arise when the fibre $F$ is disconnected.
\end{remark}

\begin{example}
Let $X$ be any compact manifold with non-negative Ricci curvature and $H^1(X;\Q) = 0$ (for example, $S^n$ $(n \geq 2)$, $\mathbb{C}P^n$, $\mathbb{R}P^{n}$ ($n$ odd), or a product thereof).
Consider the natural covering map 
$$
    p \colon \R^m \times (S^1)^n \times X 
        \longrightarrow (S^1)^m \times (S^1)^n \times X 
        = T^{m+n} \times X
$$
with fibre $\Z^m$.
Then $p^* \colon H^1(T^{m+n} \times X;\Q) \to H^1(\R^m \times (S^1)^n \times X;\Q)$ is the projection map 
$$
    \pr \colon \Q^{m+n} \to \Q^n, \qquad
    (x_1,\dots,x_m,x_{m+1},\dots,x_{m+n}) \mapsto (x_{m+1},\dots,x_{m+n}).
$$
Hence, by \Cref{thm: secat geq rank}, we obtain
$$
    \sct_1(p) 
        \geq \mathrm{rank}(p^* \colon H^1(T^{m+n} \times X;\Q) \to H^1(\R^m \times (S^1)^n \times X;\Q)) 
        = m.
$$
Moreover, $\sct_1(p) = \sct_1(\R^m \to (S^1)^m) \leq \ct_1((S^1)^m) \leq m \cdot \ct_1(S^1) = m$.
Hence, $\sct_1(p) = m$.
\end{example}

\begin{remark}
Let $X$ be a path-connected topological space and let $k$ be a field. Consider the diagonal map
\[
    \Delta \colon X\longrightarrow X\times X,\qquad x\longmapsto (x,x).
\]
Consider the induced map on cohomology $\Delta^* \colon H^1(X\times X;k)\longrightarrow H^1(X;k)$.
Since $k$ is a field, the Künneth theorem gives the isomorphism
\[
    H^1(X\times X;k)
        \cong \bigl(H^1(X;k)\otimes H^0(X;k)\bigr)
        \oplus \bigl(H^0(X;k)\otimes H^1(X;k)\bigr).
\]
As $X$ is connected, we have $H^0(X;k)\cong k$, and $H^1(X\times X;k) \cong H^1(X;k)\oplus H^1(X;k)$, where the isomorphism maps
\[
    (\alpha,\beta) \in H^1(X;k) \oplus H^1(X;k)
        \longmapsto \pi_1^*(\alpha) + \pi_2^*(\beta) 
            \in H^1(X \times X;k)\,,
\]
where $\pi_1, \pi_2 \colon X \times X \to X$ are the projection maps.
Since $\pi_1\circ\Delta = \pi_2\circ\Delta  =\id_X$, 
Hence, under the Künneth identification, the induced map is therefore simply
\[
    \Delta^* \colon H^1(X;k)\oplus H^1(X;k) \longrightarrow H^1(X;k),
        \qquad (\alpha,\beta)\longmapsto\alpha+\beta.
\]
Hence,
$
    \ker(\Delta^*) 
        = \{(\alpha,-\alpha)\mid \alpha\in H^1(X;k)\}
        \cong H^1(X;k).
$
Therefore, 
\[ 
    \mathrm{rank} (\ker(\Delta^*)) 
        = \mathrm{rank} (H^1(X;k))
        = b_1(X).
\]    
Hence, \Cref{thm: secat geq rank} does not yield anything interesting for $\TC^1(X)$ since we always have $\TC^1(X) \geq \ct_1(X) \geq b_1(X)$ for a space $X$ satisfying the assumptions of \Cref{thm: secat geq rank}.
\end{remark}

Using the Cheeger–Gromoll splitting, Bochner’s theorem was improved by Oprea as follows:

\begin{theorem}[{\cite[Theorem 3.3]{MR1866039}}]
\label{Oprea}
Let $B$ be a closed $n$-manifold with non-negative Ricci curvature such that $\pi_1(B)$ is infinite. 
If $\widetilde{B} \cong T^k \times N$ is a Cheeger--Gromoll splitting of $B$, then
$$
    b_1(B) \leq \ct(B)-\cl(B).
$$
Moreover, $b_1(B)=\ct(B)$ if and only if $B$ is a torus.
\end{theorem}

\begin{remark}
\label{rem: betti number of finite cover}
The proof of \Cref{Oprea} relies on the fact that the homomorphism 
$$
    \pi^* \colon H^*(\widetilde{B};\mathbb{Q}) \longrightarrow H^*(B;\mathbb{Q})
$$ 
induced by the finite covering map $\pi \colon \widetilde{B} \to B$ is a split monomorphism.
This follows from the existence of the corresponding transfer map in cohomology, see \cite[Remark 3.1]{MR1866039}. 
Consequently, $b_1(B)\leq b_1(\widetilde{B})$. 
\end{remark}

We now prove the $m$-category version of \Cref{Oprea}.

\begin{theorem}
\label{thm: betti num leq m-cat - m-cup}
Let $B$ be a compact manifold with non-negative Ricci curvature and infinite fundamental group. 
If $\widetilde{B} \cong T^k \times N$ is a Cheeger-Gromoll splitting of $B$, then
$$
    b_1(B) \leq \mcat(B) - \cl_{R}^{m}(N)
$$
for all $m$ and for any commutative ring $R$.
Moreover, the following are equivalent:
\begin{enumerate}
\item[(1)] $b_1(B)=\mcat(B)$ for all $m$.

\item[(2)] $b_1(B)=\mcat(B)$ for some $m \geq \dim(N)$.

\item[(3)] $B$ is a torus.
\end{enumerate}
\end{theorem}

\begin{proof}
First, note that $\cl_R^m(T^k) = k$ for all $m \geq 1$.
Hence, by \Cref{rem: betti number of finite cover}, it follows that
$$
    b_1(B) 
        \leq b_1(\widetilde{B}) 
        = b_1(T^k \times N)
        = k 
        = \cl_{R}^{m}(T^k).
$$
Therefore, by \Cref{prop: m-cat of cover leq m-cat of base}, \Cref{prop: cohomological lower bound on m-cat} and \Cref{prop: properties of m-cup length} (4), we obtain
$$
    \mcat(B) 
         \geq \mcat(\widetilde{B})
         \geq \cl_{R}^{m}(\widetilde{B})
         = \cl_{R}^m(T^k) + \cl_{R}^m(N)
         \geq b_1(B) + \cl_{R}^m(N).
$$ 
This proves the desired inequality.

Now we prove the second part of the statement.
Note that $\ct_m(T^k) = \ct(T^k) = k$ for all $m$, by \Cref{prop: m-cat = cat if X is ''aspherical''}.
Hence, it is enough to show that $(2) \implies (3)$.
Suppose that $b_1(B) = \ct_m(B)$ for some $m \geq \dim(N)$.
This implies $\cl_{\Q}^m(N) = 0$.
Since $m \geq \dim(N)$ and $N$ is a compact manifold, it follows that $\cl_{\Q}(N) = \cl^m_{\Q}(N) = 0$.
Since $N$ is simply-connected, it is orientable. 
In particular, $H^{\dim(N)}(N;\Q) = \Q$.
Hence, $\cl_{\Q}(N) = 0$ implies $\dim(N) = 0$, i.e., $N$ is a point.
\end{proof}

\begin{remark}
Since $N$ is simply-connected, it follows that $H^{1}(N;R) = 0$.
Hence, $b_1(B) \leq \ct_1(B)$, since $\cl_R^1(N) = 0$.
Therefore, for $m=1$, the above theorem does not improve upon Oprea and Strom's \cite[Theorem 5.7]{MR2879375}.
However, the following example shows that higher values of $m$ can provide a better bound than the result of Oprea and Strom.
\end{remark}

\begin{example}
Consider the space $B = T^k \times \R P^n$, where $k\geq 1$ and $n \geq 2$.
Then $\widetilde{B} = T^k \times S^n$, and hence $N = S^n$.
Since $T^k$ is aspherical, it follows that $\ct_m(T^k) = \ct(T^k) = k$ for all $m$, by \Cref{prop: m-cat = cat if X is ''aspherical''}.
Moreover, $\ct_m(T^k \times \R P^n) \leq \ct_m(T^k) + \ct_m(\R P^n) = k+n$ for all $m$.
Let $m_0$ be an integer satisfying $m_0 \geq n$.
Then $\cl_R^{m_0}(S^n) = 1$, and hence we obtain
$$
    \ct_{m_0}(B) - \cl^{m_0}_R(N) \leq k+n-1 < k+n = \ct_1(B).
$$
Thus, \Cref{thm: betti num leq m-cat - m-cup} yields a strictly better bound than \cite[Theorem 5.7]{MR2879375} in this case.
\end{example}


\section{Acknowledgments}
R.S. Arora would like to acknowledge IISER Pune - IDeaS Scholarship and Siemens-IISER Ph.D. fellowship for economical support. S. Datta acknowledges the UGC junior research fellowship. N. Daundkar gratefully acknowledges the support of DST–INSPIRE Faculty Fellowship (Faculty Registration No. IFA24-MA218), as well as Industrial Consultancy and Sponsored Research (IC\&SR), Indian Institute of Technology Madras for the New Faculty Initiation Grant (RF25261395MANFIG009294). G. C. Dutta thanks the National Board for Higher Mathematics (NBHM) for Grant No. 0204/16(10)/2024 /R\&D-II/6762.

\bibliographystyle{plain} 
\bibliography{references}
\end{document}